\documentclass[11pt]{amsart}
\usepackage{amsthm, amsfonts, amsxtra, amssymb, amscd}
\usepackage[english]{babel} 

\usepackage[colorlinks, linktocpage, citecolor = blue, linkcolor = blue]{hyperref}
 
\newtheorem*{lemma*}{Lemma}
\newtheorem{lemma}[subsection]{Lemma}
\newtheorem*{theorem*}{Theorem}
\newtheorem{theorem}[subsection]{Theorem}
\newtheorem*{proposition*}{Proposition}
\newtheorem{proposition}[subsection]{Proposition}
\newtheorem*{corollary*}{Corollary}

\newtheorem{corollary}[subsection]{Corollary}

\theoremstyle{definition}
\newtheorem*{definition*}{Definition}
 
\newtheorem{definition}[subsection]{Definition}
\newtheorem*{example*}{Example}
\newtheorem{example}[subsection]{Example}

\theoremstyle{remark}
\newtheorem*{remark*}{Remark}
 
\newtheorem{remark}[subsection]{Remark}

\DeclareMathOperator{\SL}{SL}
\DeclareMathOperator{\GL}{GL}
\DeclareMathOperator{\Sym}{Sym}

\DeclareMathOperator{\Hom}{Hom}
\DeclareMathOperator{\id}{id}
\DeclareMathOperator{\im}{im}

\newcommand{\CC}{\mathbb C}
\newcommand{\NN}{\mathbb N}
\newcommand{\QQ}{\mathbb Q}
\newcommand{\PP}{\mathbb P}
\newcommand{\ZZ}{\mathbb Z}
\newcommand{\pd}[2]{\dfrac{\partial#1}{\partial#2}}
\newcommand{\OOO}{\mathcal O}
\newcommand{\SSS}{\mathcal S}

\newcommand{\BBB}{\mathcal B}
\newcommand{\Sk}{\SSS_{k}}
\newcommand{\Sn}{\SSS_{n}}
\newcommand{\name}[1]{\textsc{#1\/}}
\newcommand{\Cplus}{\mathbb C^{+}}
\newcommand{\Cst}{\mathbb C^{*}}
\renewcommand{\phi}{\varphi}

\newcommand{\be}{\begin{enumerate}}
\newcommand{\ee}{\end{enumerate}}

\newcommand{\ps}{\par\smallskip}
\newcommand{\lex}{\ell_{\text{\it exp}}}
\newcommand{\lmon}{\ell_{\text{\it mon}}}
\newcommand{\Cblam}{\CC[\bar\lambda_{1},\ldots,\bar\lambda_{n-1}]}
\newcommand{\br}{\boldsymbol{r}}
\newcommand{\bh}{\boldsymbol{h}}
\newcommand{\bk}{\boldsymbol{k}}
\newcommand{\bn}{\boldsymbol{n}}
\newcommand{\blam}{\bar\lambda}
\newcommand{\bae}{\bar e}

\newcommand{\E}{\boldsymbol{E}}
\newcommand{\D}{\boldsymbol{D}}
\newcommand{\simto}{\xrightarrow{\sim}}

\begin{document}

\title{Perpetuants: a lost treasure}
\author{Hanspeter Kraft and Claudio Procesi}
\date{version 11 from 12.11.2019}
\address{Departement Mathematik und Informatik,
Universit\"at Basel,  Spiegelgasse~1, CH-4051 Basel}
\email{hanspeter.kraft@unibas.ch}
\address{Dipartimento di Matematica, G. Castelnuovo,
Universit\`a di Roma La Sapienza, piazzale A. Moro,  00185,
Roma, Italia}
\email{procesi@mat.uniroma1.it}

\maketitle 
\hfill{\it\small To the memory of  Gian-Carlo Rota}
\vskip1cm

 \begin{abstract} The purpose of this paper is 
to discuss the classical, and forgotten, notion of {\em perpetuants}, see Definition \ref{Perp}, and in particular to exhibit a basis of these elements  in Theorem~\ref{basis-perp}, thus closing an old line of investigation started by \name{J. J. Sylvester} in 1882.

In order to do this we also give  a proof of the classical Theorem of \name{Stroh} computing their dimensions.
\end{abstract}

\section*{Introduction}
{\em Perpetuant} (see Definition~\ref{Perp}) is one of the several concepts invented by \name{J. J. Sylvester} in his investigations of {\em covariants} for binary forms. 

One of the main goals of classical invariant theorists  was to exhibit a {\em minimal set of generators or ``Groundforms''}  for the rings of invariants under consideration, in particular for covariants of binary forms. This proved soon to be a formidable task achieved only for forms of degree up to 6.  
Perpetuants  are  strictly connected to the quest of a minimal set of generators for  a {\em limit algebra $S$ of covariants},  defined below  in Formula~\ref{limit}.    
 
The simplest description of $S$, but not very instructive,  is as the subalgebra of the polynomial ring  $R=\CC[a_0,a_1,a_2,\ldots]$  in the infinitely many variables $a_i$, $i=0,\ldots,\infty$   which is the kernel of the derivation $\D = \sum_{i=1}^\infty a_{i-1}\frac{\partial}{\partial a_i}$.

To the best of our knowledge  such an explicit description was not achieved. With our method we shall in fact exhibit such a minimal set of generators  which we call a {\em basis of perpetuants}. This is our main new result, Theorem~\ref{basis-perp}.
\ps
The term {\em perpetuant}  appears in one of the first issues of the American Journal of Mathematics \cite{Sy1882On-Subvariants-i.e} which \name{Sylvester} had founded a few years before.  A name  which will hardly appear in a mathematical paper of the last 70 years due to the complex history of invariant theory which was at some time declared dead only to resurrect several decades later. 

We learned of this word from \name{Gian-Carlo Rota}  who pronounced it with an enigmatic smile. In fact, in  \cite{KuRo1984The-invariant-theo} he laments that  {\em ``This area is in a particularly sorry state.''}
\ps
 
We were surprised to find an entry in Wikipedia where one finds useful information, but the wrong paper of \name{Stroh} is quoted.

In this entry it is mentioned that {\it\name{MacMahon} conjectured and \name{Stroh} proved the following result:}

\begin{theorem}[\cite{St1890Ueber-die-symbolis}]\label{main}
The dimension of the space of perpetuants of degree $n>2$ and weight $g$ is the coefficient of $x^g$  in
$$
{\frac {x^{2^{n-1}-1}}{(1-x^{2})(1-x^{3})\cdots (1-x^{n})}}.
$$
For $n=1$ there is just one perpetuant, of weight 0, and for $n=2$ the number is given by the coefficient of $x^g$ in $x^2/(1-x^2)$. 
\end{theorem}  
 
In order to prove our main Theorem~\ref{basis-perp} we need first to review in modern language \name{Stroh}'s proof which is quite remarkable and in a way already {\em very modern}, see Theorem~\ref{str}. The basic new idea here is to understand \name{Stroh}'s mysterious ``Potenziante'' as a dualizing tensor.
\smallskip

For a history of these  ideas and the contributions of \name{Cayley} and \name{Hammond} we refer to \name{MacMahon}  \cite{Ma1894The-Perpetuant-Inv}. More about perpetuants can be found in \cite{Ma1884On-Perpetuants,Ma1885A-Second-Paper-on-,Ma1885Memoir-on-Seminvar,Gr1903On-Perpetuants,Gr1903Types-of-Perpetuan,Wo1904On-the-Irreducibil,Wo1905Alternative-Expres,YoWo1905Perpetuant-Syzygie,Wo1907On-the-Reducibilit,Yo1924Ternary-Perpetuant,Gi1927The-Minimum-Weight}.
\ps
\subsection*{Organization of the paper}  The paper is divided into four sections.
\ps
Section 1 establishes the basic notation and recalls some standard techniques from classical invariant theory.
\ps
Sections 2 and 3 form the bulk of the paper. In section 2 we give an explicit basis of $S$ and a proof of \name{Stroh}'s Theorem. In section 3 we prove the Main Theorem giving a basis of the perpetuants.
\ps
Finally, section 4 is an appendix, explaining the role of this material in the classical theory of binary forms. We also explain a direct approah which can be used to calculate a minimal set of generating covariants.
 
\ps
\section{Back to \texorpdfstring{$19^{th}$}{19th} century}
\subsection{Semi-invariants and covariants\label{seeco}}  
One has to start with the classical notion of {\em semi-invariant}. The name is probably due to \name{Cayley} (see \cite{Sy1882On-Subvariants-i.e}, cf. \cite{Fa1876Theorie-des-Formes}), but today, with this name, we understand a different notion, so that we will use the term {\it $U$-invariant}. 

Consider the $n+1$-dimensional vector space  $P_{n}=P_n(x)\subset \CC[x]$ of  polynomials of degree $\leq n$ in the variable $x$.  

On this acts the additive group $\Cplus$ by 
\[
p(x)\mapsto p(x-\lambda)\text{ for } \lambda \in \Cplus \text{ and }p(x)\in P_{n}.
\]
As usual this action extends to an action of $\Cplus$ as automorphisms of the algebra $\OOO(P_{n})$ of polynomial functions on $P_n$.
\begin{definition}\label{semi}
The {\it algebra $S(n)$ of  $U$-invariants} of polynomials of degree $n$ is the subalgebra of the algebra of polynomial functions on $P_n$ which are invariant under the action of the group $\Cplus$:
$$
S(n) := \OOO(P_{n})^{\Cplus}.
$$
\end{definition}

The symbol $U$ is justified since, as we shall see,  the space $P_n$ can be identified with the space of {\em binary forms}, that is homogeneous polynomials of degree $n$ in two variables, over which acts the group $\SL(2,  \CC)$. 

The action of $\Cplus$ should be understood as the action of the unipotent subgroup $U$ of $\SL(2,  \CC)$,
$$
U:=\Big\{ \begin{bmatrix} 1&a\\0&1 \end{bmatrix} \mid a\in  \CC\Big\}.
$$  

\begin{remark}\label{cov}  
For the invariant theorists  of the $19^{th}$ century $S(n)$ is an {\em avatar}  of covariants of binary forms of degree $n$, a basic tool to compute invariants.  
We will explain later what this means.
\end{remark}

The operator of derivative $\frac{d}{dx}$  maps  $P_n$ surjectively  to $P_{n-1}$  commuting with the actions of 
$\Cplus$.

This induces an inclusion of duals  $P_{n-1}^*\subset P_n^*\subset P_{n+1}^*\ldots$, hence an inclusion of the rings $\OOO(P_{n})$ of polynomial functions on $P_{n}$, and finally an inclusion $S(n)\subset S(n+1)$ of $U$-invariants. We thus obtain a limit ring
\begin{equation}\label{limit}
S=\bigcup_{n=0}^\infty S(n), \text{ \ the algebra of $U$-invariants.}
\end{equation}
 
In order to have a more concrete description of $S$ 
one  needs to keep the same coordinates for the duals.  
It is then necessary  to write a polynomial $p(x)$ as a sum of {\em divided powers}, setting
\begin{equation}\label{ddp}
x^{[i]}:=\frac{x^i}{i!},\quad p(x):=\sum_{j=0}^n a_j x^{[n-j]}=\sum_{j=0}^n b_j x^{ n-j },\quad b_j:=\frac {a_j}{(n-j)!}. 
\end{equation}
Then  the operator of derivative $\dfrac d{dx}$  acts as $\dfrac d{dx}x^{[i]}=x^{[i-1]}$, hence
$$
\frac d{dx} p(x) = \frac d{dx}\sum_{j=0}^n a_j x^{[n-j]}=\sum_{j=0}^{n-1} a_j x^{[n-1-j]}.
$$ 
It follows that the coordinates $a_0,a_1,\ldots ,a_{n-1}$  which are a basis of  $P_{n-1}^*$ are mapped to the same coordinates in  $P_{n}^*$.

Therefore,  the algebra $R=\mathbb C[a_0,\ldots,a_n,\ldots]$ of polynomials  in the infinitely many variables $a_i$, $i=0,\ldots,\infty$, is the union of the algebras $R_n=\OOO(P_{n})$ of   polynomials  on the spaces $P_n$,  and  the algebra $S=R^U$ is the ring of invariants of this infinite polynomial algebra under the action of $U=\Cplus$.

A  basic feature of divided powers is that,  in the {\em binomial formula}, the binomial coefficients disappear: 
\begin{equation}\label{disb}
(a+b)^{[i]}=\frac{(a+b)^i}{i!}=\sum_{j=0}^i \frac 1{i!} \binom ij a^{i-j}b^j= \sum_{j=0}^i  a^{[i-j]}b^{[j]}.
\end{equation}
We  thus get, for $\lambda\in\Cplus$,
$$
\lambda\cdot x^{[i]}:=\frac {(x-\lambda)^i}{i!}=\sum_{j=0}^i (-1)^j  \lambda^{[j]}x^{[i-j]},
$$
and so
\begin{equation}\label{azio}
\begin{split}
\lambda\cdot f(x):=\lambda\cdot \sum_{j=0}^n a_j x^{[n-j]}
&= \sum_{j=0}^n a_j  \sum_{h=0}^{n-j} (-1)^h \lambda^{[h]} x^{[n-j-h]}\\
&= \sum_{k=0}^n   \left(\sum_{j+h=k}  (-1)^h a_j \lambda^{[h]}\right)  x^{[n-k]}.
\end{split}
\end{equation}
By duality the action of $\lambda$ on the coefficients is that $a_k$ is transformed into the $k^{th}$ coefficient of $(-\lambda)\cdot p(x)$,  that is 
\begin{equation}\label{azio1}
\lambda\cdot a_k = \sum_{j+h=k}   a_j \lambda^{[h]} =\sum_{j =0}^k  a_j \lambda^{[k-j]},
\end{equation}
for instance 
\begin{gather*}
\lambda\cdot a_{0} = a_0,\ \  \lambda\cdot a_{1} = a_0\lambda +a_1,\ \  
\lambda\cdot a_{2} = a_0\lambda^{[2]} +a_1\lambda +a_2,\\
\lambda\cdot a_{3} = a_0\lambda^{[3]}+a_1\lambda^{[2] }+a_2\lambda+a_3,\ldots.
\end{gather*}
We thus see in an explicit way that the dual action on polynomials in $a_0,a_1,\ldots$ is defined in a way independent of $n$. 

The map $\lambda\mapsto \lambda\cdot f(a_0,a_1,\ldots)=f(\lambda\cdot a_0,\lambda\cdot a_1,\ldots) $ is an additive 1-parameter group of  automorphisms with infinitesimal generator given by
\begin{equation*}
\begin{split}
\frac d{d\lambda}f(\lambda\cdot a_0,\lambda\cdot a_1,\ldots)|_{\lambda=0} &=
\sum_i\pd{}{a_i} f(a_0,a_1,\ldots) \frac d{d\lambda} ( \lambda\cdot  a_i) |_{\lambda=0}\\
&=  \sum_{i=1}^\infty a_{i-1}\pd{}{a_i} f(a_0,a_1,\ldots),
\end{split}
\end{equation*}
because  $\dfrac d{d\lambda} ( \lambda\cdot  a_i) |_{\lambda=0}=a_{i-1}$.
One deduces the next theorem which--we believe--is due to \name{Cayley}:
\begin{theorem}\label{diff} 
The algebra $S$ of $U$-invariants is formed by the polynomials $f$ in the infinitely many variables 
$a_0,a_1,a_{2},\ldots$, satisfying
$$
\D f:= \sum_{i=1}^\infty a_{i-1}\pd{}{a_i} f(a_0,a_1,a_{2}, \ldots)=0,\quad \D = \sum_{i=1}^\infty a_{i-1}\pd{}{a_i} .
$$
\end{theorem}

\ps
\subsection{Weight}\label{wei}
In classical literature the {\it weight\/}  is a way of counting in a monomial in the $a_i$  the sum of the indices $i$ appearing. That is $a_i$ has weight $i$   and $\prod_ja_j^{h_j}$  has weight $\sum_j h_jj.$ The use of this is in the following.
\begin{definition}\label{isob} 
A polynomial $f(a_0,a_{1},\ldots)$ is {\em isobaric of weight $g$}  if all the monomials appearing in $f$ have weight $g$. 
\end{definition}  
Of course, every polynomial $f(a_0,\ldots,a_n)$  decomposes in a unique way into the sum of homogeneous and isobaric components. 
\begin{definition}\label{deis} 
The algebra $R=\CC[a_0,a_{1}, \ldots]$  decomposes into the direct sum of its components $R_{n,g}$  formed by 
the  homogeneous polynomials of degree $ n$ and weight $g$. This is a bigrading as algebra.
\end{definition} 
Of course, the polynomials in $R_{n,g}$  depend only on the variables   $a_0,  a_1,\ldots,a_g$. 

\begin{remark}\label{gma}
If $n\geq g$, then  one has  $R_{n,g}=a_0^{n-g} R_{g,g}. $
\end{remark}
This definition in modern language is that of characters of a torus. It applies to $U$-invariants due to the following considerations. 
The multiplicative group  $\Cst$ of nonzero complex numbers acts by automorphisms on polynomials $f(a_0,a_1,\ldots)$ 
by $\mu\cdot a_i=\mu^i a_i$. Then a polynomial  $  f(a_0,a_1,\ldots)$ is isobaric of weight $g$ if and only if 
$$
\mu\cdot    f(a_0,a_1,\ldots)=\mu^g    f(a_0,a_1,\ldots).
$$  
It is easy to see that this action by $\Cst$ has the following commutation with the operator $\D:=\sum_{i=1}^\infty a_{i-1}\pd{}{a_i}$:
$$
\mu\cdot \D f = \mu^{-1}(\D (\mu\cdot f)).
$$
As a consequence, since $\D(a_i)=a_{i-1}$, we have the following result.
\begin{lemma}\label{mapis} 
The operator  $\D :=\sum_{i=1}^\infty a_{i-1}\pd{}{a_i}$ maps polynomials of degree $n$ and weight $g$ into polynomials of degree $n$ and weight $g-1$.
\end{lemma}

By Theorem~\ref{diff}, this implies the next result.
\begin{corollary}\label{isob.cor} 
The isobaric components of a $U$-invariant are also $U$-invariants.
\end{corollary}

\begin{definition}\label{anw}
We denote by  $S_{n}$  the subspace of  $U$-invariants homogeneous of degree $n$ and by $S_{n,g} \subset S_{n}$  the ones isobaric of weight $g$:
\begin{equation*}
S_{n}= \bigoplus_{g=0}^{\infty} S_{n,g}, \quad S=\bigoplus_{n=0}^{\infty} S_{n} = \bigoplus_{n,g \in \NN} S_{n,g}.
\end{equation*}
\end{definition}

\ps
\subsection{Reducible elements}  
In the old literature  an invariant of some positive degree $k$ is called {\em  reducible} if it is   equal to a polynomial in invariants of strictly lower degree.

This  idea of course applies to an element of any commutative graded algebra, but
today it is an unfortunate expression, since in commutative algebra   
reducible means something else. So we shall use the word {\em decomposable}.

Nevertheless, the invariant theorists  of the $19^{th}$ century used this idea in order to understand {\em a minimal set of generators} for a ring of invariants. 

In modern terms, if $A=\oplus_{i=0}^\infty A_i$ is a commutative graded algebra, with $A_0=F$ the base field and setting $I:=\oplus_{i=1}^\infty A_i$, we know that a minimal set of generators for $A$ is a basis of $I$ modulo $I^2$.  So in order to describe such minimal bases one has to describe $I/I^2$, or rather to describe a {\em complementary space to $I^2$ in $I$}. 
This, of course, is not canonical, and in fact it is interesting to read some disputes between \name{Sylvester} and \name{Fa\`a di Bruno} about the best choice of representatives.

\ps
In our situation, we have for each $n$ the algebra $S(n)$ and the  corresponding maximal ideal $I_n$. 
For general $n$ little is known about $I_{n}/I_{n}^{2}$.  In fact, one of the high points of the theory was to prove that $I_n/I_{n}^{2}$ is finite dimensional. This is \name{Gordan}'s famous Finiteness Theorem \cite{Go1868Beweis-dass-jede-C}.

But  what  was quickly discovered is that an element of $S(n)$  which is indecomposable in $S(n)$ need not remain indecomposable in $S(n+1)$. In other words the maps $I_n/I_n^2\to I_{n+1}/I_{n+1}^2$ need  not   be injective,  or also, a minimal set of generators for $S(n)$ cannot be completed to one for $S(n+1)$. As an example we will see in section~\ref{U-invar.sec}  that the generator $D$ for $S(3)$  given by Remark~\ref{noperp}  is decomposable in $S(4)$.

\ps
We are now ready to give the main definition of this paper, that of {\em perpetuant} (Sylvester  \cite{Sy1882On-Subvariants-i.e}).
\begin{definition}\label{Perp}
A {\em perpetuant} is an indecomposable element of $S(n)$ which remains indecomposable in all $S(k), \ k\geq n$.
\end{definition}
In other words  it gives an element of $I_n/I_n^2$  which {\em lives forever}, that is it remains nonzero in all $I_k/I_k^2,\ \forall\  k\geq n$.  In this sense it is {\em perpetuant}.

Of course a  perpetuant is just an indecomposable element of the limit algebra $S$. Thus to describe perpetuants is strictly related to describe minimal sets of generators for the graded algebra $S$.  In other words, denoting by $I \subset S$ the maximal homogeneous ideal   of $S$ we want to describe $I/I^2$.  

This space decomposes into a direct sum
\begin{equation*}
I/I^2=\bigoplus_{n,g\in\NN}P_{n,g}
\end{equation*} 
with $P_{n,g}$ the image of $S_{n,g}$ (the elements in $I$  of degree $n$ and weight $g$).
We may, for convenience and abuse, refer to this space as  {\em the space of perpetuants}. 
\begin{definition}\label{spapp}
  A bigraded subspace of $I$  which is a complement of $I^2$  will be called   {\em a space of perpetuants}.
\end{definition} Observe that,  while the decomposable elements,  i.e. elements of  $I^2$,  as well as the perpetuants,   i.e the elements of  $I\setminus I^2$, are intrinsic objects, a {\em space of perpetuants} that is a complement of $I^2$ in $I$  is not intrinsic, but depends on the choice of some basis of $I$,  completing a  basis of $I^2$. \smallskip

The   Theorem of Stroh gives the generating function for the dimensions of the isobaric components of $I/I^{2}$,  and can be stated as
\begin{equation*}
\sum_{g=0}^
\infty\dim(P_{n,g}) \, x^g=\begin{cases}{\displaystyle\frac {x^{2^{n-1}-1}}{(1-x^{2})(1-x^{3})\cdots (1-x^{n})}}
&\text{for } n>2, \\
x^2/(1-x^2)& \text{for }n=2, \\
1 & \text{for } n=1.
\end{cases}
\end{equation*}
In Theorem~\ref{str}  we will  construct a complement of $I^2$ in $I$, see Theorem~\ref{basis-perp}, thus giving a possible solution to the problem posed by \name{Sylvester} of describing the spaces of perpetuants.
 
\ps
\subsection{Umbral calculus}  
This is one of the forgotten parts of old invariant theory, but it really is an anticipation of some aspects of tensor calculus.
The problem arises in  the computation of invariants of some group $G$ of linear transformations on a vector space $W$ (cf. \cite{Ro1999Two-turning-points} and \cite{RoTa1994The-Classical-Umbr}).

\subsubsection*{First step: polarization and restitution}
The first step is general. The homogeneous polynomial functions on $W$ of degree $k$ can be {\em fully polarized} giving rise, in characteristic 0, to a $G$-equivariant isomorphism between this space and the space of multilinear and symmetric functions in $k$ copies of  $W$. In formulas
$$
\OOO(W)_{k}=\Sym^k(W^{*})\simeq ({W^{*}}^{\otimes k})^{\Sk}
$$
where $\Sym^{k}(W^{*})$ is the $k^{th}$ symmetric power of $W^{*}$, and $\Sk$ is the symmetric group in $k$ letters acting in the obvious way on  ${W^{*}}^{\otimes k}$.

Given a homogeneous polynomial function $f(w)$ on $W$ of degree $k$, its {\em polarizations} $f_{\alpha}$ are obtained from the expansion 
\begin{equation}\label{pol.eq}
f(t_{1}w_{1}+\cdots+t_{k}w_{k})=\sum_{\substack{\alpha=(\alpha_{1},\ldots,\alpha_{k})\\\alpha_{1}+\cdots+\alpha_{k}=k}} 
t^\alpha f_\alpha(w_1,\dots,w_k), \quad
t^\alpha := t_{1}^{\alpha_{1}}\cdots t_{k}^{\alpha_{k}}.
\end{equation}
The {\it full polarization} $P f$ is the coefficient of the product $t_{1}\cdots t_{k}$:
$$
P f (w_1,\dots,w_k) := f_{1,\ldots,1}(w_1,\ldots,w_k).
$$
It is multilinear and symmetric.
In order to obtain the inverse map, called {\em restitution}, one sets all the $w_i$ equal to $w$, i.e. $F(w_{1},\ldots,w_{n})\mapsto F(w,\ldots,w)$. Starting with $f = f(w)$ and setting $w_{i}=w$ in  equation \eqref{pol.eq} we find
$$
f((t_{1}+\cdots+t_{k})w )=(t_{1}+\cdots+t_{k})^k f(w )= \sum t^\alpha f_\alpha(w ,\dots,w ).
$$
In particular, $P f(w,\ldots,w) = k! f(w)$  (the presence of the factor $ k!$   explains why this procedure works well only in characteristic zero).

\subsubsection*{Second step: umbrae of invariants}
Next assume that $W$ is some tensor representation of $\GL(V)$ for some space $V$. For instance, in the classical literature one finds $W=\OOO(V)_{n}=\Sym^n(V^*)$,  the space of homogeneous polynomial functions of degree $n$ on $V$. Inside $W$ one then has some special vectors, usually the {\em decomposable vectors}  which span $W$ and are stable under $\GL(V)$. For instance,  for $W=\Sym^n(V^*)$ we have the functions $\phi ^n$,  the $n^{\text{\it th}}$ powers of the linear forms $\phi\in V^*$, or their divided powers $\phi^{[n]}$. 

Then a linear function  on $W$ restricted to the elements $\phi ^{[n]}$, $\phi\in V^*$, gives rise  to a homogeneous polynomial of degree $n$  on $V^*$, and this establishes an isomorphism, between $W^*$ and   the space of  homogeneous polynomials of degree $n$  on $V^*$,        which is $\GL(V)$-equivariant. Similarly, a multilinear function on $k$ copies of $W$ restricted to $ \phi_1^{[n]},\ldots,\phi_k^{[n]}$, $\phi_j\in V^*$, gives rise  to a homogeneous polynomial of degree $n$  in each of the $k$ variables $\phi_i\in  V^*$, and this also establishes a $\GL(V)$-equivariant isomorphism between $(W^{\otimes k})^*$ and   the space of  multihomogeneous polynomials of degree $n$  on  $k$  vector variables in $V^*$ (compatible with the two actions of $S_k$). 

Thus a multilinear invariant  of  $k$  copies of $W$  under some subgroup $G\subset \GL(V)$ is encoded in a polynomial  in $k$  linear forms $\phi_1,\ldots,\phi_k\in V^*$,  called {\em umbrae},  which is multihomogeneous of degree $n$, symmetric in the variables $\phi_1,\ldots,\phi_k$, and invariant under $G$.  

\subsubsection*{The symbolic method}
Combining the two steps above we obtain a map, denoted by $\E$,  which associates to    a multihomogeneous symmetric invariant of degree $n$ of $k$ {\em umbral}--variables in $V^{*}$, a  homogeneous invariant of degree $k$  on $W$, 
after interpretation as multilinear polynomial in $k$  variables in $W$ and  setting all variables equal. 

This is the basis of the {\em symbolic method  for  binary forms} where $W:=\Sym^{g}(\CC^{2})=\CC[x,y]_{g}$ which we will describe now.

\ps
\subsection{Symbolic method for binary forms}
For binary forms an explicit algorithmic way  (which extends of course to forms in any number of variables) is the following. 
A linear form is an expression of type $\alpha_1x+\alpha_2y$. With the notation of \eqref{ddp}, we have 
$$
(\alpha_1x+\alpha_2y)^{[g]}=\sum_{i=0}^g \alpha_1^{[g-i]}\alpha_2^{[i]} x^{[n-i]}y^{[i]}:=\sum_{i=0}^g a_i x^{[g-i]}y^{[i]}.
$$
Let $\phi_i:= \alpha_{1,i}x+\alpha_{2,i}y$, $i=1,\ldots,n$, be $n$ linear forms. A polynomial homogeneous of degree $g$ in each  of the $\phi_i$  is a linear combination of monomials of type 
$\prod_{i=1}^n \alpha_{1,i}^{g-r_i}\alpha_{2,i}^{r_i }$ where $0\leq r_{i}\leq g$. Then the corresponding function on $W=\CC[x,y]_{g}$ under the map $\E$ is given by
\begin{equation}\label{lE}
\E\colon \prod_{i=1}^n \alpha_{1,i}^{[g-r_i]}\alpha_{2,i}^{[r_i ]}\mapsto \prod_{i=1}^n a_{r_i}.
\end{equation}
The $\SL(2,\CC)$-invariants of  the  forms $\phi_i$, $i=1,\ldots,k$ are generated by the quadratic invariants:
\begin{equation*}
(i,j)=(\phi_i,\phi_j):=\det \begin{bmatrix} \alpha_{1,i}& \alpha_{2,i}\\
\alpha_{1,j}& \alpha_{2,j}
\end{bmatrix}= \alpha_{1,i}  \alpha_{2,j}- \alpha_{1,j}  \alpha_{2,i}.
\end{equation*} 
Therefore, the space of invariants of degree $k$  of binary forms of degree $n$ is spanned by the evaluation $\E$ of the symbols  $\prod_t(i_t,j_t),\ i_t,j_t\in\{1,\ldots,k\}$  in which each of the indices $1,\ldots,k$ appears exactly $n$  times.

The first problem of this method is to  exhibit a list of symbols which give a basis of the corresponding invariants. But the main difficulty  is to understand which symbols correspond to  decomposable invariants. 

A large number of papers of $19^{th}$ century invariant theory is devoted to these problems culminating with \name{Gordan}'s proof of finite generation, see \cite{Go1868Beweis-dass-jede-C}, \cite{Go1987Vorlesungen-uber-I}. 

A  complete answer is not known, and probably too complex to be made explicit. For this reason Theorem~\ref{main} of \name{Stroh}, and our main result, Theorem~\ref{basis-perp}, are quite remarkable.

\begin{remark}
\be
\item 
The map $\E$  is a well defined linear homomorphism from the space of polynomials  in the $n$ variables $\phi_i$  with coordinates $ \alpha_{1,i},\alpha_{2,i}$ and homogeneous of degree $g$ in each of these variables to the space of polynomials of degree $n$  in the variables $a_0,\ldots,a_g$:
$$
\E\colon \CC[\alpha_{1,1}, \alpha_{2,1},\ldots,\alpha_{1,n},\alpha_{2,n}]_{(g,g,\ldots,g)} 
\to \CC[a_{0},\ldots,a_{g}]_{n}.
$$
\item 
In the classical literature one may see a statement as
$$
\phi_1\cong\phi_2\cong\ldots\cong\phi_n 
$$ 
to mean that the map $\E$  takes the same values when permuting the $k$  umbrae $\phi_i$.
More precisely, the map $\E$ is an isomorphism when restricted to symmetric polynomials  in the $k$ variables $\phi_i$.
\item 
Finally, $\E$ is equivariant with respect to the action of $\GL(2,\CC)$ induced by the action on the variables $x,y$.
\ee
\end{remark}\label{ugg}

Formula~\eqref{lE}  suggests, when considering $U$-invariants,  to work  only with  the variables $\alpha_{i}:=\alpha_{2,i}$ and to replace the map $\E$ with 
\begin{equation}\label{lE1}
\E\colon \prod_{i=1}^n   \alpha_{ i}^{[r_i ]}\mapsto \prod_{i=1}^n a_{r_i},\quad (r_i\geq 0).
\end{equation} Notice that, in this formula the $\alpha_i$ which do not appear, i.e. for which $r_i=0$, contribute each to a factor $a_0$.

It follows that $\E$  maps linearly the space of   polynomials in $\alpha_{1},\ldots, \alpha_{n}$ to the space of polynomials homogeneous of degree $n$   in the variables $a_0,a_{1},a_{2},\ldots$,
$$
\E\colon \CC[\alpha_{1},\ldots,\alpha_{n}] \to \CC[a_{0},a_{1},a_{2},\ldots], \ 
\alpha_{1}^{[r_{1}]}\cdots \alpha_{n}^{[r_{n}]} \mapsto a_{r_{1}}\cdots a_{r_{n}},
$$ 
where a homogeneous polynomial of degree $g$ is mapped to an isobaric polynomial of weight $g$ and  homogeneous  of degree $n$. 
The map $\E$  commutes with the permutation action on the $\alpha_{ i}$, $i=1,\ldots,n$, 
e.g.
$$
\E(\alpha_1^{[3]}\alpha_2^{[2]})=\E(\alpha_3^{[3]}\alpha_1^{[2]}) =a_{0}^{n-2}a_2a_3 \text{ \ and \ } \E(\alpha_i^{[2]}\alpha_j^{[2]})=a_{0}^{n-2}a_2^2.
$$

\begin{remark}\label{hom} 
The map $\E$ is not a homomorphism of algebras. Nevertheless,  if we decompose the umbrae  in two disjoint subsets $\alpha_1,\ldots,\alpha_h$ and $\alpha_{h+1},\ldots,\alpha_k$ and consider two polynomials $f(\alpha_1,\ldots,\alpha_h)$ and $g(\alpha_{h+1},\ldots,\alpha_k)$  we have
\begin{equation*}\label{moE}
\E\left(f(\alpha_1,\ldots,\alpha_h)\cdot  g(\alpha_{h+1},\ldots,\alpha_n)\right)=
\E(f(\alpha_1,\ldots,\alpha_h))\cdot \E(g(\alpha_{h+1},\ldots,\alpha_n)).
\end{equation*}
\end{remark}
In this formalism we loose the action of $\GL(2,\CC)$, but we still have the translation action by $\Cplus$, which commutes with $\E$. In term of differential operators,   we see that
\begin{equation}\label{ecd}
\E\circ \sum_{i=1}^n \pd{}{\alpha_i}= \sum_{i=1}^\infty a_{i-1}\pd{}{a_i}\circ \E.
\end{equation}

When working with $U$-invariants  the previous umbral calculus corresponds  to the further simplification of considering, instead of binary forms, polynomials in $x$, and the action is just translation $x\mapsto x-\lambda$. 
One may further restrict to monic polynomials, and  so use the formula
\begin{equation*}\label{UMb1}
( x+\alpha  )^{[g]}=\sum_{i=0}^g  \alpha ^{[i]} x^{[g-i]} = x^{[n]}+\sum_{i=1}^n a_i x^{[g-i]} .
\end{equation*}
Given any polynomial in the $a_1,a_{2},\ldots$ which is invariant under  translation and of given weight  one can reconstruct a corresponding  $U$-invariant by making it homogeneous  by inserting powers of $a_0$. 

\par\medskip
\section{Stroh's potenziante and duality}
\subsection{Potenziante} Following 
\name{Stroh} we define  the {\em potenziante} by
\begin{equation}\label{pot0}
\pi_{n,g}:=\E\left((\sum_{j=1}^n\lambda_j\alpha_j)^{[g]}\right)=
\sum_{\substack{r_{1},\ldots,r_{n}\in \NN\\r_{1}+\cdots+r_{n}=g}} \lambda_{1}^{r_{1}}\cdots\lambda_{n}^{r_{n}} \, 
a_{r_{1}}\cdots a_{r_{n}}
\end{equation}   
where the $\alpha_1\ldots,\alpha_{n}$ are all umbrae, and one uses  Formula~\eqref{lE1}. Moreover,
\begin{equation*}
\sum_{g=0}^\infty \pi_{n,g}=\E(\exp { \sum_{j=1}^n\lambda_j\alpha_j}).
\end{equation*}
Notice that $\pi_{n,g}$ is a polynomial in the variables $\lambda_1,\ldots,\lambda_n$ and the variables $a_0,\ldots,a_g$.  Sometimes we need to change the variables $\lambda_{j}$, so we shall stress this dependence by writing:
$$
\pi_{n,g}=\pi_{n,g}(\lambda;a) = \pi_{n,g}(\lambda_1,\ldots,\lambda_n;a_0,a_1,\ldots,a_g). 
$$
By construction, $\pi_{n,g}(\lambda;a)$  is homogeneous of degree $n$ and isobaric of weight $g$ in the $a_{i}$, and symmetric and of degree $g$ in the $\lambda_j$.  So it can be developed  in term of the symmetric functions of degree $g$  in  the $\lambda_j$. 

\begin{definition}\label{sgu} 
Denote by $\Sigma_{n,g}=\CC[\lambda_{1},\ldots,\lambda_{n}]_{g}^{\Sn} \subset \CC[\lambda_{1},\ldots,\lambda_{n}]$ the subspace of {\it symmetric polynomials\/}  
in $\lambda_1,\ldots,\lambda_n$ which are homogeneous of degree $g$.
\end{definition} The space $\Sigma_{n,g} $  has several useful bases, all indexed by partitions:
$$  h_{1}\geq h_{2}\geq\cdots\geq h_{n} \geq 0,\ h_i\in \mathbb N,\quad h_{1}+\cdots +h_{n}=g.$$

We first take as basis of $\Sigma_{n,g}$ the {\it total monomial sums  $m_{h_1,\dots,h_n}$}, i.e. the sum over the $\Sn$-orbit of   $\lambda_{1}^{h_{1}}\cdots \lambda_{n}^{h_{n}}$ ($\Sn$ the symmetric group on $n$ elements) where 
$  h_{1}\geq h_{2}\geq\cdots\geq h_{n} \geq 0$ and  $h_{1}+\cdots +h_{n}=g$:
$$
m_{h_{1},\ldots,h_{n}}(\lambda) :=\sum_{\Sn\text{-orbit}}\sigma(\lambda_{1}^{h_{1}}\cdots\lambda_{n}^{h_{n}}).
$$
It follows from the definition of $\pi_{n,g}(\lambda;a)$ (see formula~\eqref{pot0}) that the monomial 
$a_{h_{1}}\cdots a_{h_{n}}$ appears in $\pi_{n,g}$ with coefficient $m_{h_{1},\ldots,h_{n}}(\lambda)$.
From this we easily get the following result.

\begin{proposition}
\be
\item In $\pi_{n,g}(\lambda;a)$ the total monomial sum  $m_{h_1,\dots,h_n}(\lambda)$ has as coefficient 
the product  $a_{h_1} a_{h_{2}}\cdots a_{h_{n}}$:
$$
\pi_{n,g}(\lambda;a) =\E\left((\sum_{r=1}^n\lambda_r\alpha_r)^{[g]}\right)
=\sum_{\substack{ h_1\geq \ldots\geq h_n \geq 0\\ h_{1}+\cdots+h_{n}=g}} m_{h_1,\dots,h_n} 
a_{h_1} a_{h_{2}} \cdots a_{h_{n}}. 
$$
\item These coefficients form a basis   of the space $\CC[a]_{n,g}$ of homogeneous polynomials in $a_0,  a_1,a_2,\ldots$ of degree $n$  and  weight $g$.
\ee
\end{proposition}\label{tms}

\begin{example}\label{g3} With $n=g=3$ we find
\begin{multline*}
(\lambda_1\alpha_1+\lambda_2\alpha_2+\lambda_3\alpha_3)^{[3]}=\lambda_1^3\alpha_1^{[3]}+ \lambda_2^3\alpha_2^{[3]}+ \lambda_3^3\alpha_3^{[3]}+ 
\lambda_1^2\lambda_2 \alpha_1^{[2]}\alpha_2+ \lambda_1^2\lambda_3 \alpha_1^{[2]}\alpha_3 \\ +\lambda_2^2\lambda_1 \alpha_2^{[2]}\alpha_1+\lambda_2^2\lambda_3 \alpha_2^{[2]}\alpha_3+\lambda_3^2\lambda_1 \alpha_3^{[2]}\alpha_1+\lambda_3^{2}\lambda_2 \alpha_3^{[2]}\alpha_2  + \lambda_1  \lambda_2  \lambda_3   \alpha_1 \alpha_2\alpha_3.
\end{multline*}
Applying  $\E$ we get
$$
(\lambda_1^3 + \lambda_2^3 + \lambda_3^3)a_0^2 a_3  + (\lambda_1^2\lambda_2  +\lambda_1^2\lambda_3  +\lambda_2^2 \lambda_1 +\lambda_2^2\lambda_3+\lambda_3^2\lambda_1+\lambda_3^2\lambda_2) a_0 a_1a_2   +   \lambda_1  \lambda_2  \lambda_3   a_1 ^3
$$
which is equal to
$$
m_{3,0,0}(\lambda) a_0^2 a_3 + m_{2,1,0}(\lambda) a_{0}a_{1}a_{2} + m_{1,1,1}(\lambda) a_{1}^{3}.$$
With $n=2$ and $g=4$ we find
$$
(\lambda_1\alpha_1+\lambda_2\alpha_2)^{[4]}=\lambda_1^4\alpha_1^{[4]}+ \lambda_2^4\alpha_2^{[4]}+  (\lambda_1^3\lambda_2 \alpha_1^{[3]}\alpha_2 +\lambda_2^{3}\lambda_1 \alpha_2^{[3]}\alpha_1)+  \lambda_1^2\lambda_2^2  \alpha_1^{[2]}\alpha_2^{[2]}.
$$
Applying  $\E$ this gives 
$$
(\lambda_1^4 + \lambda_2^4 ) a_0^3a_4  + (\lambda_1^3\lambda_2   +\lambda_2^3\lambda_1 )a_0^2a_1a_3 +    \lambda_1^2\lambda_2^2a_2^2,
$$
which is equal to
$$
m_{4,0}(\lambda) a_{0}^{3}a_{4} + m_{3,1}(\lambda) a_{0}^{2}a_{1}a_{3} + m_{2,2}(\lambda) a_{2}^{2}.
$$
Finally, with $n=3$ and $g=2$ we have
\begin{multline*}
(\lambda_1\alpha_1+\lambda_2\alpha_2+\lambda_3\alpha_3)^{[2]}=\lambda_1^2\alpha_1^{[2]}+ \lambda_2^3\alpha_2^{[2]}+ \lambda_3^3\alpha_3^{[2]}\\+ 
\lambda_1\lambda_2 \alpha_1\alpha_2+ \lambda_1\lambda_3 \alpha_1\alpha_3  +\lambda_2\lambda_3 \alpha_2\alpha_3\ .
\end{multline*}
Applying  $\E$ we get
$$
(\lambda_1^2 + \lambda_2^2 + \lambda_3^2)a_0 ^2a_2  + (\lambda_1\lambda_2  +\lambda_1 \lambda_3 +\lambda_2\lambda_3) a_0 a_1^2  , 
$$
which is equal to
$$
m_{2,0,0}(\lambda) a_0^2 a_2 + m_{1,1,0}(\lambda) a_{0}a_{1}^2.$$
Notice that all coefficients are divisible by $a_0$ as expected, because $n>g$.  
\end{example}

\ps
\subsection{Duality}
Proposition~\ref{tms} can be understood in the following way. The tensor 
\begin{equation*}
\pi_{n,g}(\lambda;a) \in \Sigma_{n,g}\otimes\CC[a]_{n,g} \subset \CC[\lambda_{1},\ldots,\lambda_{n}]^{\Sn}\otimes_{\CC}\CC[a_{0},\ldots,a_{g}]
\end{equation*}
defines a {\it duality\/} between the subspace $\Sigma_{n,g}\subseteq \CC[\lambda_{1},\ldots,\lambda_{n}]^{\Sn}$ of homogeneous symmetric polynomials of degree $g$ in $n$ variables, and the subspace $\CC[a]_{n,g}\subseteq\CC[a_{0},\ldots,a_{g}]$ of polynomials, homogeneous  of degree $n$ and with weight $g$. \smallskip

In general, given two finite dimensional vector spaces $U,W$  and denoting by  $U^{\vee}, W^{\vee}$ their duals one has the canonical isomorphisms
\begin{equation}\label{dut}
U\otimes W\simeq \Hom(U^{\vee},W)\simeq \Hom(W^{\vee},U).
\end{equation} 
E.g., if $\pi \in U\otimes W$ is a tensor, then the corresponding map $\pi\colon U^{\vee} \to W$ is given by 
$\phi\mapsto(\phi\otimes \id_{W})(\pi)$.

A {\em dualizing tensor $\pi\in U\otimes W$} is an element which corresponds, under these isomorphisms, to an isomorphism $U^{\vee}\simto  W$ (or $W^{\vee}\simto U$). Thus, a dualizing tensor $\pi$  equals, for any basis $u_1,\ldots,u_k$ of $U$, to $\pi=\sum_{i=1}^k u_i\otimes w_i$, where $w_1,\ldots,w_k$ is a basis of $W$. 

\begin{definition}
Let $\pi\in U \otimes W$ be a dualizing tensor.
If $M \subseteq U$ is a subspace, then the {\it orthogonal subspace\/} $M^{\perp} \subseteq W$ is defined to be the image of $(U/M)^{\vee}$ in $W$ under the isomorphism $U^{\vee} \simto W$ corresponding to $\pi$.\end{definition}

Choosing a basis $(u_{i})_{i=1}^{n}$ of $U$ such that $(u_{j})_{j=1}^{m}$ is a basis of $M$ and writing $\pi = \sum_{i}u_{i}\otimes w_{i}$, then $(w_{k})_{k=m+1}^{n}$ is a basis of $M^{\perp}$. Moreover, the image of $\pi$ in $U/M \otimes W$ defines a dualizing tensor $\bar\pi \in U/M \otimes M^{\perp}$.

\begin{remark}\label{sdua}
In general, a tensor $\pi\in U\otimes W$ gives, via the isomorphism~\eqref{dut}, two maps
\begin{equation}\label{ippi}
\pi_1\colon W^{\vee}\to U,\quad \pi_2:U^{\vee}\to  W. 
\end{equation} 
and we have the following:
\be
\item
$\pi_2=\pi_1^{\vee}$, $\im(\pi_1)=\ker(\pi_2)^\perp$, $\im(\pi_2)=\ker(\pi_1)^\perp$.
\item
$\pi\in \im(\pi_1)\otimes \im(\pi_2)$ is a dualizing tensor for these two spaces.
\item 
If  $\pi=\sum_{i=1}^k u_i\otimes w_i$, and if the $u_i$ are linearly independent, then the  image $\im(\pi_2)$ of the corresponding map is the span of the elements $w_i$. Similarly, if the $w_{i}$ are linearly independent.
\item 
If $U'\subseteq U$ is a subspace and $\pi\in U'\otimes W$, then the  associated subspace in $W$  is the same when computed with $U$ or $U'$.
\ee
\end{remark}
\begin{proof}[Sketch of Proof]
Let $\pi=\sum_{i=1}^ku_i\otimes w_i$ with $k$ minimal. Then clearly both the $u_i$ as well as the $w_j$ are linearly independent, and they span the two spaces $\im(\pi_1)$ and $\im(\pi_2)$. Now the various claims follow easily.
\end{proof}

Applying this to the dualizing tensor 
$$\pi_{n,g} \in 
\Sigma_{n,g}\otimes_{\CC}\CC[a_{0},\ldots,a_{g}]_{n,g}
$$ 
we obtain, from Proposition~\ref{tms}, the following two isomorphisms:
\begin{align*}
\Sigma_{n,g}^{\vee} \simto \CC[a]_{n,g}\colon& 
\phi\mapsto (\phi\otimes \id)(\pi_{n,g})=    
(\phi\otimes \id)(\textstyle{\sum}_{h} m_{h}\otimes a_{h})=
\textstyle{\sum}_{h}\phi(m_{h})a_{h},\\
\CC[a]_{n,g}^{\vee} \simto \Sigma_{n,g}\colon & 
\psi\mapsto (\id \otimes \psi)(\pi_{n,g})=   (\id \otimes \psi)(\textstyle{\sum} _{h} m_{h}\otimes a_{h})=
\textstyle{\sum}_{h}\psi(a_{h})m_{h},
\end{align*}
As a consequence of the remark above this gives the next result.

\begin{proposition}\label{LaD}
For every subspace $M \subset \Sigma_{n,g}$ we have the orthogonal subspace $M^{\perp} \subset \CC[a]_{n,g}$, and the duality between $\Sigma_{n,g}/M$ and $M^{\perp}$ is given by the image of $\pi_{n,g}$ in $\Sigma_{n,g}/M\otimes \CC[a_{0},\ldots,a_{g}]_{n,g}$. 

Given a basis  $u_1,\ldots,u_N$ of $\Sigma_{n,g}$ such that $u_1,\ldots,u_m$ is a basis of $M$, and writing $\pi_{n,g}=\sum_{j=1}^{N}u_j\otimes b_j$, then  the elements $b_{m+1},\ldots, b_N$  form a basis of $M^\perp.$
\end{proposition}

Now consider some homogeneous ideal $J \subset\QQ[\lambda_1,\ldots,\lambda_n]$  and the corresponding quotient map $\QQ[\lambda_1,\ldots,\lambda_n]\to  \QQ[\lambda_1,\ldots,\lambda_n]/J$. By restriction, we have a homomorphism
$$
\QQ[\lambda_1,\ldots,\lambda_n]^{\Sn} \to \QQ[\lambda_1,\ldots,\lambda_n]/J
$$ 
with kernel  $\QQ[\lambda_1,\ldots,\lambda_n]^{\Sn}\cap J$.  Denoting the images of the $\lambda_i$  by $\bar \lambda_i$  we get an image of the potenziante 
\begin{equation*}
\pi_{n,g}(\bar\lambda_{1},\ldots,\bar\lambda_{n}) = \E\left((\sum_{j=1}^n\bar\lambda_j\alpha_j)^{[g]}\right)
\in \CC[\bar\lambda_{1},\ldots,\bar\lambda_{n}]\otimes_{\CC}\CC[a_{0},\ldots,a_{n}],
\end{equation*} 
and we obtain the following result.

\begin{proposition}\label{bass} 
If $(B_{i})_{i\in I_{g}}$ is a basis of $\Sigma_{n,g}/\Sigma_{n,g}\cap J$ and if
$$
\pi_{n,g}(\bar \lambda_{1},\dots,\bar\lambda_{n})=\sum_{i\in I_{g}}B_{i}\otimes A_{i},
$$
then $(A_{i})_{i\in I_{g}}$ is a basis of the subspace of $\CC[a]_{n,g}$ orthogonal to 
$\Sigma_{n,g}\cap J.$
\end{proposition}

\ps
\subsection{Description of the \texorpdfstring{$U$}{U}-invariants}
We have seen in Theorem~\ref{diff} that the $U$-invariants $S=\CC[a_{0},a_{1},\ldots]^{U}$ form the kernel of the differential operator $\D :=\sum_{i=1}^\infty a_{i-1}\pd{}{a_i}$. Moreover, $S$ is bigraded by degree and weight: $S = \bigoplus_{n,g \in \NN}S_{n,g}$, see Definition~\ref{anw}.
\subsubsection*{A very remarkable formula}
From
\begin{equation*}
(\sum_{i=1}^n \pd{}{\alpha_i})(\sum_{r=1}^n \lambda_r\alpha_r)^{[g]}=(\sum_{i=1}^n \lambda_i)(\sum_{r=1}^n \\\lambda_r\alpha_r)^{[g-1]}
\end{equation*}
we obtain, by Formula~\eqref{ecd},
\begin{equation}\label{pot3}
\begin{split}
\sum_{i=1}^\infty a_{i-1}\pd{}{a_i} \pi_{n,g} &=(\sum_{i=1}^n \lambda_i) \, \pi_{g-1,n},
\quad\text{or}\\
\D (\E(\exp{\sum_{j=1}^n\lambda_j\alpha_j})) &=e_1(\lambda)\,\E(\exp{\sum_{j=1}^n\lambda_j\alpha_j}). 
\end{split}
\end{equation}

\begin{remark}\label{tras}
The meaning of this formula is that, using the duality  between symmetric functions in $n$ variables and polynomials in the $a_i$ of degree $n$, the transpose of the operator $\D $  is the multiplication by $\sum_{i=1}^n \lambda_i$.
\end{remark}
The remarkable formula~\eqref{pot3} gives us a straightforward way of describing  both,  the operator $\D  $ and also a basis of the $U$-invariants.

For this we change the basis of the space $\Sigma_{n,g}$ of symmetric functions from the  total monomial sums  $m_{h_1,\dots,h_n}$ to the monomials $e_1^{k_1}\dots e_n^{k_n}$, $\sum_{j} jk_j=g$,
where   $e_{i}=e_{i}(\lambda_{1},\ldots,\lambda_{n})$ is the $i^{\text{\it th}}$ elementary symmetric function.

 Expressing the total monomial sums $m_{h_{1},\ldots,h_{n}}$ in the $e_{i}$'s we get, for some $\alpha_{h_1,\dots,h_n, k_1,\dots,k_n}\in\mathbb Z$:
\begin{equation}\label{moel}
m_{h_1,\dots,h_n}=\sum_{  k_1,\dots,k_n} \alpha_{h_1,\dots,h_n, k_1,\dots,k_n}e_1^{k_1}\dots e_n^{k_n}. 
\end{equation}

The potenziante $\pi_{n,g}$ in this new basis  appears as 
\begin{equation*}
\pi_{n,g}=\sum_{\substack{0\leq k_1,\ldots,k_n  \\  k_{1}+2k_{2}+\cdots+n k_{n}=g}}
e_1^{k_1}\dots e_n^{k_n}\tilde U_{k_1,\ldots,k_n},
\end{equation*}
where the new elements  $\tilde U_{k_{1},\ldots,k_{n}}$ are given by
\begin{equation}\label{tilU}
\tilde U_{k_{1},\ldots,k_{n}}=\sum_{h_1,\dots,h_n} \alpha_{h_1,\dots,h_n, k_1,\dots,k_n} \prod_{j=1}^n a_{h_j},
\end{equation} 
and also form a basis of $\CC[a]_{n,g}$. 
By formula~\eqref{pot3} we get
\begin{equation*}
\D \pi_{n,g}=\sum_{\substack{k_1,\ldots,k_n\geq 0\\  \sum ik_i=g}}e_1^{k_1}\dots e_n^{k_n}
\D \tilde U_{k_1,\ldots,k_n}
=\sum_{\substack{j_1,\ldots,j_n\geq 0\\  \sum i j_i=g-1}} e_1^{j_1+1}\dots e_n^{j_n}
\tilde U_{j_1,\ldots,j_n}=e_1\pi_{n,g-1}
\end{equation*}
which implies the following result.

\begin{corollary}
\be
\item 
We have
\begin{equation}\label{DU}
\D\tilde U_{k_1,\ldots,k_n}=\begin{cases}
0\quad &\text{if }k_1=0\\
\tilde U_{k_1-1,\ldots,k_n} &\text{if } k_1>0.
\end{cases}
\end{equation}
\item 
The elements $U_{k_{2},\ldots,k_{n}}:=\tilde U_{0,k_2,\ldots,k_n}$   form a basis of the space $S_{n,g}$ of the $U$-invariants of degree $n$ and weight $g$. 
\ee

{\em It is interesting to remark that these results hold over $\ZZ$ and not just over $\CC$.}

\end{corollary}\label{uinv}
In an alternative way  we can impose the relation
$\sum_{r=1}^n\lambda_r=0$ and denote by $\bar \lambda_r$ the class of $\lambda_r$ modulo $ \sum_{r=1}^n\lambda_r=0$. Denote by $\bar \Sigma_n$  the algebra of symmetric functions in the 
 $\bar \lambda_r$ which is a polynomial algebra over the elements $e_2(\bar \lambda),\ldots,e_n(\bar\lambda)$.  The space  of symmetric functions of degree $g$ in $\bar\lambda_{1},\ldots,\bar\lambda_{n}$,
$$
\bar\Sigma_{n,g}=(\CC[\lambda_{1},\ldots,\lambda_{n}]^{\Sn}/(\lambda_{1}+\cdots+\lambda_{n}))_{g} = \CC[\bar\lambda_1,\ldots,\bar\lambda_n]_{g}^{\Sn}
$$ 
has as basis the monomials of weight $g$ in the  elements $\bar e_2:=e_2(\bar \lambda),\ldots,\bar e_n:=e_n(\bar\lambda)$. This implies the following result.

\begin{proposition}\label{ssee}
\be
\item
The potenziante 
$\bar\pi_{n,g}( \bar\lambda;a)\in \bar\Sigma_{n,g}\otimes \CC[a]_{g}$ has the form
\begin{equation}\label{redpo}
\bar\pi_{n,g}( \bar\lambda_1,\ldots,\bar\lambda_n;a_0,a_1,\ldots,a_g)=
\sum_{\substack{k_2,\ldots,k_n\geq 0\\  \sum ik_i=g}} 
\bar e_2^{k_2}\dots\bar e_n^{k_n} U_{k_2,\ldots,k_n}.
\end{equation}
\item
In particular, $\bar\pi_{n,g}(\bar\lambda,a)$ is a dualizing tensor between the space $\bar\Sigma_{n,g}$  of symmetric functions of degree $g$ in $ \bar\lambda_1,\ldots,\bar\lambda_n$  and the space $S_{n,g}$ of $U$-invariants of degree $n$ and weight $g$. 
\item
The elements $U_{k_{2},\ldots,k_{n}}=\tilde U_{0,k_2,\ldots,k_n}$ of formula
\eqref{tilU} form a basis of the space $S_{n,g}\subset \CC[a_{0},\ldots,a_{n}]$ of $U$-invariants  of degree $n$ and weight $g$,  dual to the basis $\bar e_2^{k_2}\dots\bar e_n^{k_n}$  of  $\bar\Sigma_{n,g}$.
 \ee
\end{proposition}

As a corollary we have
\begin{equation*}
\sum_{g=0}^
\infty\dim(S_{n,g})x^g=\frac{1}{(1-x^{2})(1-x^{3})\cdots (1-x^{n})}\ .
\end{equation*}

\begin{remark}\label{upc}
\name{Cayley} and \name{MacMahon} use the word {\em non unitariants}  for various symmetric functions in the differences of the roots, due to the fact that they are indexed by partitions with no part of size 1. 
\end{remark}

As a consequence of the proposition above and of Proposition~\ref{LaD} we see that the quotients of $\bar\Sigma_{n,g}$ correspond to subspaces of $S_{n,g}$.
So our final task is to identify the subspace $O_{n,g}$ of $\bar\Sigma_{n,g}$  orthogonal to the subspace   of decomposable elements of $S_{n,g}$ and from that a choice of perpetuants.

\ps
\subsection{Decomposable \texorpdfstring{$U$}{U}-invariants} 
For a given $h \in \NN$, $1\leq h <n$ we have \eqref{disb}:
$$
(\lambda_{1}\alpha_{1}+\cdots+\lambda_{n}\alpha_{n})^{[g]} = 
\sum_{j=0}^{g} (\lambda_{1}\alpha_{1}+\cdots+\lambda_{h}\alpha_{h})^{[j]} (\lambda_{h+1}\alpha_{h+1}+\cdots+\lambda_{n}\alpha_{n})^{[g-j]}
$$
which implies, by Remark~\ref{hom}, the following decomposition of the potenziante:
\begin{equation}\label{depo}
\pi_{n,g}(\lambda_{1},\ldots,\lambda_{n};a) =
\sum_{j=0}^{g} \pi_{h,j}(\lambda_{1},\ldots,\lambda_{h};a)\cdot
\pi_{n-h,g-j}(\lambda_{h+1},\ldots,\lambda_{n};a).
\end{equation}
Consider the ideal $J_h \subset \CC[\lambda_1,\ldots,\lambda_n]$  generated by the two linear elements $\lambda_1+ \cdots+\lambda_h$ and 
$\lambda_{h+1}+ \cdots+\lambda_n$ (or the ideal $\bar J_h\subset \CC[\bar\lambda_1,\ldots,\bar\lambda_n]$  generated by $\bar\lambda_1+ \cdots+\bar\lambda_h$). Then
\begin{multline*}
\CC[\lambda_{1},\ldots,\lambda_{n}]/J_{h} = \CC[\bar\lambda_1,\ldots,\bar\lambda_n]/\bar J_{h} = \\
\CC[\lambda_{1},\ldots,\lambda_{h}]/(\lambda_{1}+\cdots+\lambda_{h}) \otimes_{\CC}
\CC[\lambda_{h+1},\ldots,\lambda_{n}]/(\lambda_{h+1}+\cdots+\lambda_{n}),
\end{multline*}
and the image of $\CC[\lambda_{1},\ldots,\lambda_{n}]^{\Sn}$ is contained in 
\begin{equation*}
\CC[\lambda_{1},\ldots,\lambda_{h}]^{\SSS_{k}}/(\lambda_{1}+\cdots+\lambda_{h}) \otimes_{\CC}
\CC[\lambda_{h+1},\ldots,\lambda_{n}]^{\SSS_{n-k}}/(\lambda_{h+1}+\cdots+\lambda_{n}).
\end{equation*}
Denote by $T_{n,g,h}$ the subspace  of this tensor product, formed by  homogeneous elements of degree $g$, and let
$\bar\Sigma_{n,g,h} \subseteq T_{n,g,h}$ be the image of 
$\bar\Sigma_{n,g}\subset \CC[\bar\lambda_1,\ldots,\bar\lambda_n]$.
Let us  write $\tilde\lambda_i$ for the class of $\bar\lambda_i$  modulo  $\bar J_h$, and 
consider the image $\bar\pi_{n,g,h}$ of $\bar \pi_{n,g}$  modulo $\bar J_h$.  We get from Formula~\eqref{depo}:
\begin{equation}\label{svil1}
\bar\pi_{n,g,h}
=\sum_{j=0}^g\pi_{h,j}(\tilde\lambda_1,\ldots,\tilde\lambda_h;a)\cdot
\pi_{n-h,g-j}(\tilde\lambda_{h+1},\ldots,\tilde\lambda_n;a)
\end{equation}  
as an element from $\bar\Sigma_{n,g,h}\otimes S_{n,g}\subset T_{n,g,h}\otimes S_{n,g}$.

\begin{lemma}\label{decomp} With the notation above we have the following results.
\be
\item[(a)]
If $(B_{i})_{i\in I}$ is a basis of  $\bar\Sigma_{n,g,h}$ and
$$
\bar\pi_{n,g,h}= \sum_{i\in I}B_{i}\otimes A_{i},
$$
then the $A_{i}$ form a basis of the  $U$-invariants decomposable as:
$$S_{n,g,h}:=\sum_{j}^{g} S_{h,j}\cdot S_{n-h,g-j}.$$
In particular, the potenziante $\bar\pi_{n,g,h}(\tilde\lambda;a) \in \bar\Sigma_{n,g,h}\otimes S_{n,g}$ is a dualizing tensor between $\bar\Sigma_{n,g,h}$ and the space $S_{n,g,h}\subset  S_{n,g}$.
\item[(b)]
In the correspondence between subspaces of $S_{n,g}$ and of $\bar\Sigma_{n,g}$  given by the potenziante $\bar\pi_{n,g}$ (see Proposition~\ref{ssee}(2)) the subspace $S_{n,g,h}$ is the orthogonal to $\bar\Sigma_{n,g}\cap \bar J_{h}$.
\ee
\end{lemma}

\begin{proof}
(a) Developing Formula~\eqref{svil1} for $\bar\pi_{n,g,h}$ by using Formula~\eqref{redpo} for all terms we get
\[
\bar\pi_{n,g,h} =  \sum_{j=0}^{g}\sum_{\substack{k_{2}',\dots,k_{h}',k_{2}'',\ldots, k_{n-h}''\\
\sum_{i} ik_{i}' = j, \  \sum_{i} ik_{i}'' = g-j}}\!\!\!\!
({e_{2}'}^{k_{2}'}\cdots{e_{h}'}^{k_{h}'} \otimes {e_{2}''}^{k_{2}''}\cdots{e_{n-h}''}^{k_{n-h}''})
U_{k_{2}'\ldots,k_{h}'} U_{k_{2}'',\ldots,k_{n-h}''}
\]
where $e_{i}':=e_{i}(\tilde\lambda_{1},\ldots,\tilde\lambda_{h})$ and $e_{i}'':=e_{i}(\tilde\lambda_{h+1},\ldots,\tilde\lambda_{n})$. Now the elements $U_{k_{2},\ldots,k_{h}}$ span the space $S_{h,j}$ for $j:=\sum_{i} i k_{i}$, it follows that the elements $ U_{k_{2}'\ldots,k_{h}'} U_{k_{2}'',\ldots,k_{n-h}''}$  span $S_{n,g,h}=\sum_{j=0}^{g} S_{h,j}\cdot S_{n-h,g-j}$. 
Since the tensor products ${e_{2}'}^{k_{2}'}\cdots{e_{h}'}^{k_{h}'} \otimes {e_{2}''}^{k_{2}''}\cdots{e_{n-h}''}^{k_{n-h}''}$ are linearly independent it follows that the $A_i$ also span $S_{n,g,h}$   using parts (2) and (3) of Remark \ref{sdua}.

It follows from Proposition~\ref{ssee}(2) and Remark~\ref{sdua} that $\bar\pi_{n,g,h}$ is a dualizing tensor between  $\bar\Sigma_{n,g,h}$, the image of $\bar\Sigma_{n,g}$,  and a subspace $S'_{n,g,h}\subseteq  S_{n,g}$. By the first part   $S'_{n,g,h}= S_{n,g,h}$ and the $A_i$ are also linearly independent.
\ps
(b) By Remark~\ref{sdua}  (1) or Proposition \ref{bass} this is clear, since $\bar\Sigma_{n,g}\cap \bar J_{h}$ is the kernel of the surjective map $\bar\Sigma_{n,g}\to \bar\Sigma_{n,g,h}$.
\end{proof}

\ps
\subsection{The symmetric functions \texorpdfstring{$p_{h}$ and $q_{n}$}{ph and qh}}\label{ph-and-qn.subsec}
The space $\bar\Sigma_{n,g}\cap \bar J_{h}$ consists of the symmetric functions in $\bar\lambda_{1},\ldots,\bar\lambda_{n}$ of degree $n$ which are divisible by $\bar\lambda_1+ \cdots+\bar\lambda_h$.   Such  a symmetric function is also divisible by the elements $\bar\lambda_T:=\sum_{i\in T}\bar\lambda_i $  for all subsets $T\subset\{1,\ldots,n\}$ of cardinality $|T|=h$.

For $h<\frac n2$ consider the symmetric function  
\begin{equation*}\label{ipi}
p_h:=\prod_{1\leq j_1<j_2<\ldots<j_h\leq n}(\bar\lambda_{j_1}+\bar \lambda_{j_2}+ \cdots+ \bar\lambda_{j_h})=\prod_{\substack{T\subset\{1,2,\ldots,n\} \\ |T|=h}}\bar\lambda_T
\end{equation*} 
of degree $\binom n h$.  It follows that $p_h$  is an irreducible element of $\bar\Sigma_{n}$ and that $\bar\Sigma_{n,g}\cap \bar J_{h}$ consists of  the multiples of $p_{h}$ of degree $g$.
When $n=2h$ is even and $h>1$, then $\bar\lambda_T=-\bar\lambda_{\{1,2,\ldots,n\}\setminus T}$. Therefore,  we define
\begin{equation*}
p_h:=\prod_{1 <j_2<\ldots<j_h\leq 2h}(\bar\lambda_{ 1}+\bar \lambda_{j_2}+ \cdots+ \bar\lambda_{j_h})
=\prod_{\substack{T\subset\{1,2,\ldots,n\} \\ |T|=h,\ 1\in T}}\bar\lambda_T.
\end{equation*} We claim that $p_h$ is   irreducible,   of degree  $\frac 12\binom {2h} h$ and still symmetric for $h>1$. In fact it is clearly symmetric with respect to the permutations which fix 1 so it is enough to see the symmetry under the transposition $(1,2)$.  This fixes all factors  for which $\lambda_2=2$;  as for the product $\Pi$ of the remaining factors $\bar\lambda_T,\ 1\in T,\ 2\notin T$  it replaces 1 with 2  and maps these set of factors bijectively to the set of factors associated to sets $T$ with $1\notin T,\ 2\in T$. 

For these sets   the map $T\mapsto \{1,2,\ldots,n\}\setminus T$ is a bijection  with the factors of $\Pi$.   By formula $\bar\lambda_T=-\bar\lambda_{\{1,2,\ldots,n\}\setminus T}$ the product of $\Pi$ is thus equal to $\epsilon \Pi$ with $\epsilon=(-1)^{|\Pi|}$. Now clearly $|\Pi|= \binom{2h-2}{h-1}=2\binom{2h-3}{h-2}$  is even.
 This proves the following lemma.

\begin{lemma}\label{priml} 
For $n>2$ and $h \leq \frac n2$ the space $\bar\Sigma_{n,g}\cap\bar J_{n}$ consists  of the elements of  $\bar\Sigma_{n,g}$ which are multiples of the symmetric function $p_h$.
\end{lemma}

Let us define the following symmetric function
$$
q_{n}:=p_{1}p_{2}\cdots p_{m} \text{ where } m:=\left\lfloor\frac{n}{2}\right\rfloor.
$$
We claim that $\deg q_{n}= 2^{n-1}-1$. In fact, 
\begin{align*}
\sum_{j=1}^{\frac {n-1}2}\binom nj =\frac 12\sum_{j=1}^{ n-1 }\binom nj &= \frac 12(2^n-2)=2^{n-1}-1 \qquad
\text{if $n$ is odd,}\\
\sum_{j=1}^{h-1}\binom {2h}j+ \frac 12 \binom {2h} h &= \frac 12 2^{2h}-1=2^{n-1}-1 \qquad
\text{if $n=2h$ is even.}
\end{align*}

\begin{theorem}\label{str}
Let us  now  assume $n>2$.
With respect to the potenziante $\bar\pi_{n,g}\in \bar\Sigma_{n,g}\otimes S_{n,g}$
the space of decomposable $U$-invariants of degree $n$ and weight $g$ is the orthogonal to $O_{n,g}:=\bar\Sigma_{n,g}\cap (q_{n})$. It has as basis   
the coefficients of the potenziante in the quotient algebra  $\CC[\bar\lambda_1,\ldots,\bar\lambda_n]/(q_{n})$.
\end{theorem}

\begin{proof} 
This follows from the following argument.  In the correspondence between subspaces of $\bar\Sigma_{n,g}$ and of $S_{n,g}$ given by the potenziante $\pi_{n,g}$ (Proposition~\ref{ssee}(2)) we have seen in Lemma~\ref{decomp} that the space of  $U$-invariants decomposed as a sum of products of $U$-invariants of degree $h$ and $n-h$  is the orthogonal of the subspace of multiples of $p_h$. Hence the entire space of decomposable $U$-invariants is the orthogonal of the intersection of all the subspaces of multiples of the various $p_h$.
But these symmetric functions are all irreducible in the algebra of symmetric functions in $\bar \lambda_1,\ldots,\bar \lambda_n$, and distinct, so that this intersection is exactly  the space of multiples of $q_n$.
\end{proof}
This we believe is the main step in \name{Stroh}'s proof.
\begin{corollary}\label{iPeP}
If $M_{n,g}\subset \bar\Sigma_{n,g}$ is a complement to $O_{n,g}:=\bar\Sigma_{n,g}\cap (q_{n})$, then the orthogonal of $M_{n,g}$ in $ S_{n,g}$ is a space of perpetuants (Definition \ref{spapp}) of degree $n$ and weight $g$.
\end{corollary}
\begin{proof}
This follows by duality. The orthogonal of $M_{n,g}$ is a complement of the orthogonal of $O_{n,g}$, which is the space of decomposable elements.
\end{proof}

\ps
\subsection{Generating functions and proof of  Stroh's Theorem}
Define
$$
N_{n,g}:=\#\{2\mu_2+3\mu_3+\cdots+n\mu_n  =g\},
$$
the number of ways of partitioning $g$ with numbers between 2 and $n$. We have 
$$
\sum_{g=0}^{\infty} N_{n,g}\, x^g = \frac{1}{(1-x^{2})\cdots(1-x^{n})}.
$$ 
The dimension of  the space  $\bar\Sigma_{n,g}$ is clearly $N_{n,g}$, and the subspace of those divisible by an element of degree $i$ has dimension $N_{n,g-i}$ if $i\leq g$ and 0 otherwise. It follows that the space $O_{n,g} = (q_{n})\cap \bar\Sigma_{n,g}$ of multiples of $q_n$ has dimension   $N_{n,g- 2^{n-1}+1}$ if $g\geq 2^{n-1}-1$ and 0 otherwise. 
Now Corollary~\ref{iPeP} shows that this is the dimension of the perpetuants of degree $n>2$ and weight $g$, and we thus get for the generating function, in degree $n$:
\begin{equation*}
\sum_{g=2^{n-1}-1}^{\infty} N_{n,g- 2^{n-1}+1} \, x^g = (\sum_{k=0}^{\infty} N_{n,k} \, x^k)x^{2^{n-1}-1} = \frac{x^{2^{n-1}-1}}{(1-x^{2})\cdots(1-x^{n})}.
\end{equation*}
This proves the Main Theorem of \name{Stroh} provided we do the cases $n=1$ and $2$.
For $n=1$ we have $S_{1} = \CC a_{0}$, and so the only homogeneous perpetuant is clearly $a_0$. For $n=2$ the only decomposable elements are the multiples of $a_0^2$. We have
$$
E(\bar\lambda_1\alpha_1+\bar\lambda_2\alpha_2)^{[g]} =  
\bar\lambda_1^gE\left((\alpha_1-\alpha_2)^{[g]}\right)= 
\bar\lambda_1^g \, E\left(\sum_{j=0}^g \alpha_1 ^{[ j]}(-\alpha_2)^{[g-j]}\right), \text{ and }
$$
\begin{eqnarray*}
E\left(\sum_{j=0}^g \alpha_1 ^{[ j]}(-\alpha_2)^{[g-j]}\right)
&=& 
\sum_{j=0}^g   (-1)^{ g-j } a_ja_{g-j} = \\
&=&
\begin{cases}
0 & \text{if $g$ is odd,}\\
\sum_{j=0}^{h-1}   2(-1)^{  j } a_ja_{g-j} +(-1)^{  h}a_h^2
& \text{if $g=2h$ is even.}
\end{cases} 
\end{eqnarray*}
This shows that there is exactly one perpetuant of degree 2  in every even weight $>0$, and so the generating function is $x^2/(1-x^2)$ as claimed.

\ps
\section{A  basis of the perpetuants}
In the next paragraph   we construct an  explicit  basis for a space of perpetuants (Definition~\ref{spapp}).

\ps
\subsection{Leading exponents}
Using Corollary~\ref{iPeP}, we will now define a special basis in order to obtain a basis of the perpetuants, see Theorem~\ref{basis-perp} below. As before, we will work in the polynomial algebra 
$$
\CC[\lambda_{1},\ldots,\lambda_{n}]/(\lambda_{1}+\cdots+\lambda_{n}) = 
\CC[\bar\lambda_{1},\ldots,\bar\lambda_{n-1}]
$$ 
where $\blam_{i}$ is the image of $\lambda_{i}$. For $\br=(r_{1},\ldots,r_{n-1}) \in \NN^{n-1}$ we set $\blam^{\br}:=\blam_{1}^{r_{1}}\cdots \blam_{n-1}^{r_{n-1}}$,  so that any  $f \in \CC[\bar\lambda_{1},\ldots,\bar\lambda_{n-1}]$ can be written in the form  $f = \sum_{\text{\it finite}} c_{\br}\bar\lambda^{\br}$.

We use the usual lexicographic order $\leq$ on the exponents:
$$
(r_{1},\ldots,r_{n-1}) < (s_{1},\ldots,s_{n-1}) \ \iff \ 
r_{k}< s_{k} \text{ for } k := \min\{i \mid r_{i}\neq s_{i}\}.
$$

\begin{definition}
For a nonzero polynomial $f \in \CC[\bar\lambda_{1},\ldots,\bar\lambda_{n-1}]$, 
$f = \sum_{i} c_{\br}\bar\lambda^{\br}$,  the maximum $\br_{0}:=\max\{\br \mid c_{\br}\neq 0\}$ is called the 
{\it leading exponent} of $f$ and is denoted by $\lex(f)$. Furthermore,
$\lmon(f):= c_{\br_{0}}\blam^{\br_{0}}$ is called the {\it leading monomial} of $f$.
\end{definition}
\begin{remark}\label{mm}
For two polynomials $f,g$ we have $\lex(f\cdot g)=\lex(f)+\lex(g)$.
\end{remark}
As before, we denote by $\bae_{2},\ldots,\bae_{n} \in \Cblam$ the images of the elementary symmetric functions $e_{2},\ldots,e_{n} \in \CC[\lambda_{1},\ldots,\lambda_{n}]$.

\begin{lemma}
\be
\item
The leading exponent of $\bae^{\bh}:=\bae_{2}^{h_{2}}\cdots \bae_{n}^{h_{n}}$ is given by 
$$
\lex(\bae^{\bh})=(2(h_{2}+\cdots+h_{n}), h_{3}+\cdots+h_{n},\ldots,h_{n-1}+h_{n}, h_{n}) \in \NN^{n-1}.
$$
\item
The leading exponents of the monomials $\bae_{2}^{h_{2}}\cdots \bae_{n}^{h_{n}}$ are distinct and are formed by all sequences $(r_{1},\ldots,r_{n-1})$ with $r_{1}-2 r_{2}\in 2\NN$ and $r_{i}\geq r_{i+1}$.
\ee
\end{lemma}\label{lex-of-e^h}

\begin{proof}
(1) 
The leading monomial of $\bae_{j}$ comes from the term
\[
\begin{split}
\blam_{1}\blam_{2}\cdots\blam_{j-1}\blam_{n} &=-\blam_{1}\blam_{2}\cdots\blam_{j-1}(\blam_{1}+\cdots+\blam_{n-1}) \\
&=-\blam_{1}^{2}\blam_{2}\cdots\blam_{j-1} + \text{lower terms}.
\end{split}
\]
Therefore we have 
$$
\lex(\bae_{2}) = (2,0,\ldots,0), \ \lex(\bae_{3}) = (2,1,0,\ldots,0), \cdots,  \lex(\bae_{n}) = (2,1,\ldots,1),
$$
and the claim follows from Remark~\ref{mm}.
\ps
(2) 
This follows immediately from (1) by setting $2h_2:=r_1-2r_2,\, h_j:=r_{j-1}-r_j,\ n> j\geq 3,\ h_n=r_{n-1}$.
\end{proof}

Recall the definition of the symmetric function $q_{n}\in \CC[\blam_{1},\ldots,\blam_{n}]$ from section~\ref{ph-and-qn.subsec}:
\begin{equation*}
\begin{split}
p_h &:= \prod_{1\leq j_1<j_2<\ldots<j_h\leq n}(\bar\lambda_{j_1}+\bar \lambda_{j_2}+ \cdots+ \bar\lambda_{j_h})=\prod_{\substack{T\subset\{1,2,\ldots,n\} \\ |T|=h}}\bar\lambda_T \quad\text{for }2h<n,
\\
p_m&:=\prod_{1 <j_2<\ldots<j_m\leq 2m}(\bar\lambda_{ 1}+\bar \lambda_{j_2}+ \cdots+ \bar\lambda_{j_m})
=\prod_{\substack{T\subset\{1,2,\ldots,n\} \\ |T|=m,\ 1\in T}}\bar\lambda_T\quad\text{for }n=2m,
\end{split}
\end{equation*}
and
$$
q_{n}:=p_{1}\cdots p_{m} \text{ where }  m:=\left\lfloor\frac{n}{2}\right\rfloor.
$$
\begin{lemma}\label{lex-qn}
The leading exponent of $q_{n}$ is $\lex(q_{n})=(2^{n-2},2^{n-3},\ldots,2,1)$.
\end{lemma}
\begin{proof}
For $T:=\{  j_1,j_2,\ldots,j_h\}$,  $1\leq j_1<j_2<\ldots<j_h\leq n$ we have
$$
\blam_T=\begin{cases}
\blam_{j_1}+\blam_{j_2}+ \cdots+ \blam_{j_h} & \text{if } j_h<n\\
\blam_{j_1}+\blam_{j_2}+ \cdots+ \blam_{j_{h-1}}-\sum_{i=1}^{n-1}\blam_i &\text{if } j_h=n
\end{cases}
$$  
Thus, if $n\in T$, then $\blam_T= -\blam_{T'}$ where $T' := \{1,\ldots,n\}\setminus T$.
The map $T\mapsto T'$ is a bijection between the subsets $T$ of  $ \{1,2,\ldots,n\}$ containing $n$ and of cardinality $h$    with  the subsets   of  $ \{1,2,\ldots,n-1\}$   of cardinality $n-h$. This implies, for $2h<n$:
\begin{gather*}
p_h = \pm\prod_{\substack{T\subset\{1,\ldots,n-1\} \\ |T|=h}}\blam_T 
\prod_{\substack{T\subset\{1,\ldots,n-1\} \\ |T|=n-h }}\blam_T
=\pm \, f_h \, f_{n-h} \\
\text{ \ where \ }
f_k:= \prod_{\substack{T\subset\{1,\ldots,n-1\} \\ |T|=k}}\bar\lambda_T.
\end{gather*}
The leading term of $\blam_T,\ T\subset\{1,2,\ldots,n-1\}$ is $\blam_j$ with $j := \min T$, and  
the number of  subsets $T\subset\{1,2,\ldots,n-1\}$ with $|T|=h$ and $j=\min T$ equals the number of subsets 
$T\subseteq \{j+1,\ldots,n-1\}$ with $|T|=h-1$. This number is equal to $\binom{n-1-j}{h-1}$ if $h\leq n-j$, and $0$ otherwise. Setting $\binom m k=0$ if $m<k$, we see that the leading exponent of $f_{k}$ is given by
$$
\lex(f_{k}) = \textstyle{(\binom{n-2}{k-1},\binom{n-3}{k-1},\ldots,\binom{n-i-1}{k-1},\ldots,\binom{1}{k-1},\binom {0}{k-1})}.
$$
(Recall that $\binom{0}{0} = 1$.)
The leading exponent of $p_h$ is thus 
\begin{multline*}
\lex(p_{h}) = \textstyle{(\binom{n-2}{h-1}+\binom{n-2}{n-h-1},\ldots,
\binom {n-i-1} {h-1}+\binom {n-i-1} {n-h-1},\ldots,} \\
\textstyle{\binom{1}{h-1}+\binom{1}{n-h-1},\binom{0}{h-1}+\binom{0}{n-h-1})}.
\end{multline*}
If $n=2m+1$, then $q_{n} = p_{1}\cdots p_{m}$, and we find for the leading exponent of $q_{n}$, $\lex(q_{n})=(r_{1},\ldots,r_{n-1})$ where
$$
r_{i} = \sum_{h=1}^{m} \left(\binom{n-i-1}{h-1} + \binom{n-i-1}{n-h-1}\right)=\sum_{h=0}^{2m-1}\binom{2m-i}{h} = 2^{2m-i}=2^{n-i-1},
$$
as claimed.
\ps
If $n=2m$,  we have $q_{n}= p_{1}\cdots p_{m-1}p_{m}$ where $p_{m}=\prod_{\substack{T\subset\{1,2,\ldots,n\} \\ |T|=h,\ 1\in T}}\bar\lambda_T$. In this case, the map $T \mapsto T':=\{1,\ldots,2m\}\setminus T$ is a bijection between the subsets containing $1$ and $2m$ and of cardinality $m$ and the subsets of $\{2,3,\ldots,2m-1\}$ containing $m$ elements. Hence
$$
p_{m}=\pm \prod_{\substack{T\subset\{1,\ldots,2m-1\} \\1\in T,\, |T|=m}}\blam_T 
\prod_{\substack{T\subset\{ 2,\ldots,2m-1\} \\ |T|= m }}\blam_T.
$$
The leading monomial of the first product is $\blam_{1}^{\binom{2m-2}{m-1}}$. For the second product, we see as above that the number of subsets of $\{2,\ldots,2m-1\}$ of cardinality $m$ with minimum $j\geq 2$ is equal to $\binom{2m-j-1}{m-1}$. Hence
$$
\lex(p_{m}) = \textstyle{(\binom{2m-2}{m-1}, (\binom{2m-3}{m-1},\ldots,\binom{2m-i-1}{m-1},\ldots,\binom{1}{m-1},0)},
$$  
and thus we get for the leading exponent $\lex(q_{n}) = (r_{1},\ldots,r_{n-1})$
\begin{equation*}
\begin{split}
r_{i} &=  \sum_{h=1}^{m-1} \left(\binom{2m-i-1}{h-1} + \binom{2m-i-1}{2m-h-1}\right) + \binom{2m-i-1}{m-1}\\
&= \sum_{h=0}^{2m-2}\binom{2m-i-1}{h} = 2^{2m-i-1} = 2^{n-i-1}.
\end{split}
\end{equation*}
This proves the lemma.
\end{proof}

\begin{remark}\label{lexq=lexe}
For $n\geq 4$ we have
$$
\lex(q_{n}) = \lex(e^{\bn}) \text{ where }\bn := (0,2^{n-4},2^{n-5},\ldots,2,1,1).
$$
Moreover,  $\lex(q_{3})=(2,1)=\lex(\bae_{3})$.
\end{remark}

\ps
\subsection{A basis for the perpetuants}
Recall that $\bar\Sigma_{n,g}$ is the space of symmetric functions of degree $g$ in $\blam_{1},\ldots,\blam_{n}$,
$$
\bar\Sigma_{n,g}=(\CC[\lambda_{1},\ldots,\lambda_{n}]^{\Sn}/(\lambda_{1}+\cdots+\lambda_{n}))_{g} = \CC[\bar\lambda_1,\ldots,\bar\lambda_n]_{g}^{\Sn}.
$$
In the next lemma we use the partial order 
$$
(t_{2},\ldots,t_{n}) \succeq (s_{2},\ldots,s_{n}) \iff t_{i}\geq s_{i} \text{ for all } i.
$$
\begin{lemma}\label{scell}
For $n\geq 3$ a basis $\BBB_{n}$ of a complement of $O_{n,g}:=\bar\Sigma_{n,g}\cap (q_{n})$,  in the space of symmetric functions in $\bar\Sigma_{n,g}$, is formed by the monomials $e^{\bh}:=\prod_{k=2}^n e_k^{h_k}$   with $\sum_k kh_k=g$, satisfying
$$
\bh = (h_{2},\ldots,h_{n}) \nsucceq \bn:=(0,2^{n-4},2^{n-5},\ldots,2,1,1) \text{ for }n>3, 
$$
respectively, $\bh = (h_{2},h_{3}) \nsucceq (0,1)$ for $n=3$.
\end{lemma}
\begin{proof}
A basis of  $O_{n,g}$ is formed by the  symmetric functions  $q_{n}\bae^{\,\bk} = q_{n} \prod_{j=2}^n \bae_j^{k_j}$ with $\sum_j jk_j=g-2^{n-1}+1 $.  We have seen, in Remark~\ref{lexq=lexe}, that the leading  exponent  of $q_{n}  \bae^{\,\bk}$ equals the leading  exponent of $\bae^{\,\bn+\bk}$.  It follows that the set
\begin{equation*} 
\mathbb X_{n}:=\left\{q_{n} \bae^{\,\bk}\mid \sum_{j}j k_{j}=g-2^{n-1}+1\right\} 
\cup \left\{\bae^{\,\bh}\mid \sum_{i}ih_{i}=g,\ \bh\not \succeq \bn\right\} \subset \bar\Sigma_{n,g}
\end{equation*}
has the same leading exponents as the basis $\{\bae^{\,\bh} \mid \sum_{i}ih_{i}=g\}$ of $\bar\Sigma_{n,g}$. Since these leading exponents are distinct, by Lemma~\ref{lex-of-e^h}(2), it follows that $\mathbb X_{n}$ is a basis of  $\bar\Sigma_{n,g}$ hence  $\BBB_{n}$ is a basis of   a complement, in $\bar\Sigma_{n,g}$,  of $O_{n,g}$, hence the claim.
\end{proof}

We have seen in Proposition~\ref{ssee} that  the potenziante 
$\pi_{n,g}(\blam;a)\in \bar\Sigma_{n,g}\otimes  S_{n,g}$ has the form
\begin{equation}\label{Uk}
\pi_{n,g}( \bar\lambda_1,\ldots,\bar\lambda_n;a_0,a_1,\ldots,a_g)=
\sum_{\substack{k_2,\ldots,k_n,\\  \sum ik_i=g}} 
e_2^{k_2}\dots e_n^{k_n} U_{k_2,\ldots,k_n}
\end{equation}
where the $U_{k_{2},\ldots,k_{n}}$ form a basis of the space $S_{n,g}\subset \CC[a_{0},\ldots,a_{n}]$ of $U$-invariants  of degree $n$ and weight $g$. 

Using Corollary~\ref{iPeP} with the basis $\BBB_{n}$ of a complement $M_{n,g}\subset \bar\Sigma_{n,g}$   to $O_{n,g} $ constructed above we get as consequence our main result.

\begin{theorem}\label{basis-perp}
The elements $U_{k_{2},\ldots,k_{n}}$ from Formula~\eqref{Uk} with 
$$
\bk \succeq \bn = (0,2^{n-4},\ldots,2,1,1)
$$  
(resp. $ \bn = (0,1)$)  form a basis of a space of perpetuants of degree $n>3$ (resp. $n=3$) and weight  $g$.
\end{theorem}

Observe that the decomposable elements do not have a basis  extracted from the elements $U_{k_{2},\ldots,k_{n}}$.\begin{remark}\label{trii} 
Finally, in order to compute explicitly the perpetuants  of Theorem~\ref{basis-perp} one needs  to compute the numbers $ \alpha_{h_1,\dots,h_n, k_1,\dots,k_n}$ of Formula~\eqref{moel}. One possible algorithm is to  compute first 
\begin{equation*} 
e_1^{k_1}\dots e_n^{k_n}=\sum_{  h_1,\dots,h_n} \beta_{h_1,\dots,h_n, k_1,\dots,k_n}\, m_{h_1,\dots,h_n}.
\end{equation*}
The numbers $\beta_{h_1,\dots,h_n, k_1,\dots,k_n}$ form an upper  triangular matrix  $E+A$ of non negative integers with 1 on the diagonal, and its inverse $E-A+A^2-\ldots $ has as entries the integers $ \alpha_{h_1,\dots,h_n, k_1,\dots,k_n}$.

The integer $\beta_{h_1,\dots,h_n, k_1,\dots,k_n}$  is computed as the coefficient  of the monomial $\prod_{i=1}^n \lambda_i^{h_i}$  in the development of  $e_1^{k_1}\dots e_n^{k_n}$.
\end{remark}

\par\medskip
\section{Binary forms} {\em This is a complement to set into the $19^{th}$  century context the theory developed.}
The $q+1$-dimensional vector space  $P_{q}=P_q(x)\subset \CC[x]$ of  polynomials of degree $\leq q$ in the variable $x$, introduced in Section~\ref{seeco},   can be thought of as a {\em  non-homogeneous form} of the space, still denoted by   $P_{q}=P_q(x,y)\subset \CC[x,y]$, of  homogeneous  polynomials of degree $q$ in the variables $x,y$. These are the classical {\em binary forms} or {\em binary $q$-antics.}  On this space acts the group $\GL(2, \CC)$, and, in fact, these spaces form the list of irreducible representations of $\SL(2, \CC)$.  Of course  this is the first  case of the more general theory of  
{\em $n$-ary $q$-antics}, i.e. of  homogeneous  polynomials of degree $q$ in the $n$ variables $x_1,\ldots,x_n$. 
\ps
One of the themes of Algebra of the $19^{th}$ century was to study the algebra $R_q$ of polynomial functions on $P_q$ which are invariant under $\SL(2, \CC)$, and then try the general case of  invariants of $n$-ary $q$-antics.   In particular, to determine a minimal set of generators for such an algebra.  The question whether such a minimal set of generators is finite was one of the main problems of this period, and proved by \name{Gordan} \cite{Go1868Beweis-dass-jede-C} for binary forms  by a difficult combinatorial method.

The problem of finite generation of invariants for a general linear group action, also known as {\it \name{Hilbert}'s $14^{th}$ problem}, has now a very long and complex history (cf. \cite{Na1965Lectures-on-the-fo,Na1959On-the-14-th-probl}) with still several open questions.

In fact,  $R_q$ is a graded algebra, and if $I_q$ denotes the ideal of $R_q$  formed by elements with no constant term, the question is to study $I_q/ I_q^2.$   A partial question is to understand the graded dimension of $I_q/ I_q^2,$ which by \name{Gordan}'s Theorem is a polynomial.   There are in fact various formulas for the graded dimension of $R_q$, but for $I_q/ I_q^2,$ to our knowledge, the only known cases are those in which one can exhibit generators for $I_q/ I_q^2$.  Thus, for binary forms only a few cases are explicitly known.  It is therefore quite remarkable that for perpetuants such a formula exists.

\ps
The reason   to introduce $U$-invariants comes from the theory of {\em covariants} of binary forms, a notion introduced as a tool to compute invariants of binary forms. 
Covariants appear in  3 different forms. For more details, we refer to the literature, e.g. \cite[Chap.~15.1, Proposition~2 and Theorem~1]{Pr2007Lie-Groups---An-Ap}.

\begin{proposition}\label{cova} 
There are canonical bijections between the following objects, called covariants of $P_{n}$ of degree $k$ and order $p$.
\be
\item[(i)]
$\SL(2,\CC)$-equivariant polynomial maps $P_{n}\to P_{p}$  of degree $k$; 
\item[(ii)]
$\SL(2,\CC)$-invariant polynomials on $P_n\oplus \CC^2$ of bidegree $k,p$.
\item[(iii)]
$U$-invariants of $P_n$ of degree
$k$ and isobaric of weight $\frac{nk -p}{2}$. 
\ee
In particular an $\SL(2,\CC)$-invariant on $P_n$ of degree $k$ is a
$U$-invariant of degree $k$ and weight  $\frac{nk}{2}.$ 
\end{proposition}
\noindent
(The reader experienced in algebraic geometry may see that the geometric reason behind these statements is the fact that $\SL(2,\CC)/B \simeq \PP^1$ is compact.)

\begin{proof} 
(i)$\iff$(ii): 
Given such a polynomial map $F\colon P_n\to P_p$ we can evaluate the form $F(f)$  in a point $(x,y)\in  \CC^2$,  $\tilde F(f,(x,y)):=F(f)(x,y)$  obtaining an $\SL(2,\CC)$-invariant of the desired form. The opposite construction is essentially tautological by the definition  of the actions.
\ps
(ii)$\iff$(iii):
Observe that a regular function on  
$P_n\oplus (\CC^{2}\setminus\{0\})$ extends as polynomial   on $P_n\oplus \mathbb C^2$. Under $\SL(2,\CC)$, the space $\CC^2\setminus\{0\}$ is the orbit of  $e_1$ with stabilizer $U$. This implies that the polynomials on $P_n\oplus \CC^2$ invariant under $\SL(2,\CC)$ are in bijection with the polynomials on $P_n\times \{e_1\}$ invariant under $U$.

Now consider the torus elements  $D_t:=\left[\begin{smallmatrix} t^{-1}&0\\0&t\end{smallmatrix}\right]\in \SL(2,\CC)$. They act on the space $\CC^2$ transforming $x\mapsto  t^{-1}x$, $y\mapsto ty$. The action on the forms $f\in P_n$ is
$$
(D_tf)(x,y)=f(tx,t^{-1}y) =  \sum_{i=0}^na_i (tx)^{[n-i]}(t^{-1}y)^{[i]}=
\sum_{i=0}^na_i  t^{n-2i}x^{[n-i]}y^{[i]}.
$$ 
In other words, $D_t$ transforms $a_i\mapsto t^{n-2i}a_i$.  A covariant $F$ of degree $k$ and order $p$ 
must be an invariant function of this transformation on $P_n\oplus \mathbb C^2$, or
$$ 
F(t^na_0, \ldots,t^{n-2i}a_i,\ldots,t^{-n}a_n, t^{-1}x,t y)= F(a_0, \ldots,a_n,x,y).
$$
By assumption, $F = \sum_{i=0}^{p}F_{i}(a_{0},\ldots,a_{n})x^{p-i}y^{i}$, hence
$$
 F_0(t^na_0, \dots,t^{n-2i}a_i,\dots,t^{-n}a_n)(t^{-1}x)^p = F_0(a_0, \dots,a_n)x^{p}.
 $$
A monomial in $F_0$ in the $a_i$ is of weight $g$, hence it is multiplied by $t^{nk-2g}$. We deduce that for every
monomial we have $nk-2g-p=0$,  as required.
\end{proof}

The $U$-invariant $F_{0}$ associated to a covariant $F$ is called its {\em source} (or {\em Quelle} in German).
There is a simple formula to write down the covariant starting  from its source, see \cite{Hi1993Theory-of-algebrai}.

\par\medskip
\subsection{\texorpdfstring{$U$}{U}-invariants for binary forms}\label{U-invar.sec}
For the algebra $S(n)$ of $U$-invariants for  $P_n$ the results are not as precise as for the limit algebra $S$.

In classical literature  explicit computations  were done correctly only up to  degree 6, and degree 8, with partial results in degree 7. With the help of computers now one has computations up to degree 12.  Here we want to give a simple method which we believe is due to \name{Cayley} and which works very well up to degree 4. 
\ps
Let us take a polynomial $f=\sum_{i=0}^na_i x^{[n-i]} $ with $a_0\neq 0$. Under the
transformation $x\mapsto x-\frac{a_1}{a_0}$ it is transformed into a polynomial  with $a_1=0$ (cf. Formula~\ref{azio1}):
\begin{equation*}
\begin{split}
f(x-\frac{a_{1}}{a_{2}}) &= a_0(x-\frac{a_1}{a_0})^{[n]}+ a_1(x-\frac{a_1}{a_0})^{[n-1]}+\cdots \\
&= a_0 x^{[n]}-a_0\frac{a_1}{a_0}  x^{[n-1]}+\cdots+ a_1 x  ^{[n-1]}+\cdots \\
& = a_0 x^{[n]} + (-\frac{a_{1}^{2}}{2a_{0}} + a_{2})x^{[n-2]}+\cdots
\end{split}
\end{equation*}
More formally, let $P_n^0\subset P_{n}$ be the set of polynomials of degree $n$ with 
$a_0\neq 0$, and let $P_n'\subset P_{n}^{0}$ be the set of
polynomials of degree $n$ with $a_0\neq 0$, $a_1=0$. The previous remark shows that acting with $U$ we
have an isomorphism $U\times P_{n'}\simto P_n^0$.  Thus we have an identification of the $U$-invariant functions on
$P_n^0$ with the functions on $P_n'$.  More precisely, the map (notation from Formula~\eqref{azio})
$$
\pi\colon P_{n}^{0} \to P_{n}', \quad f\mapsto \frac{a_{1}}{a_{0}} \cdot f,
$$ 
is $U$-invariant, and so the pull-backs of the coordinate functions of $P_{n}'$ together with $a_{0}^{-1}$ generate the $U$-invariants on $P_{n}^{0}$. By Formula~\eqref{azio1}, these pull-backs are given by
$$
(-\frac{a_{1}}{a_{0}})\cdot a_{k} = \sum_{j =0}^k  a_j (-\frac{a_1}{a_0})^{[k-j]}=a_0^{-k+1}c_k
$$
where
\begin{equation*}
\begin{split}
c_k &= (- a_1 )^{[k ]}+ \sum_{j =1}^k a_0^{  j-1 } a_j (- a_1 )^{[k-j]}\\
&=(- a_1 )^{[k ]}+  a_1 (- a_1 )^{[k-1]}+\sum_{j =2}^k a_0^{  j-1 } a_j (- a_1 )^{[k-j]} \\
&=(-1)^{k}(1-k) a_1^{[k ]} +\sum_{j =2}^k (-1)^{k-j}a_0^{  j-1 } a_j a_1^{[k-j]}.
\end{split}
\end{equation*}
Thus we get the following result.
\begin{theorem} We have
$S(n)[a_0^{-1}]=\CC[c_2,\dots,c_n][a_0,a_0^{-1}]$ where $a_{0},c_{2},\ldots,c_{n}$ are algebraically independent.
In particular, $\dim S(n)=n$.
\end{theorem}
\par\noindent
Let us explicit some of these elements:
\begin{equation}\label{ck}
\begin{split}
c_2 &= -a_1^{[2]}+a_0a_2,\\
c_3 &= \phantom{-} 2 a_1^{[3]}- a_0a_1a_2+a_0^2a_3,\\
c_4 &= -3 a_1^{[4]}+ a_0a_1^{[2]}a_{2} - a_0^2a_1a_3+a_0^3a_4,\\
c_5 &= \phantom{-} 4 a_1^{[5]} - a_0 a_1^{[3]} a_{2} + a_0^2 a_1^{[2]} a_{3} - a_0^3 a_{1} a_4 + a_0^4 a_5,\\
c_6 & =-5a_1^{[6]} + a_0 a_1^{[4]} a_{2} - a_0^2 a_1^{[3]} a_{3} + a_0^3 a_{1}^{[2]} a_4 - a_0^4 a_1 a_5 + a_0^5 a_6.
\end{split}
\end{equation}
By construction, $c_k$ is a $U$-invariant of degree $k$ and weight $k$ (cf. Definition~\ref{wei}).

\begin{corollary}\label{peu} 
The subalgebra of $S(n)$ generated  by the $U$-invariants with weight equal to the degree is the polynomial ring $\CC[c_2,\dots,c_n]$.
\end{corollary}
 
\ps
\subsection{An algorithm}
If we want to understand $U$-invariants from these formulas it is necessary to compute the intersection
\begin{equation}\label{intS}
S(n)=\mathbb C[c_2,\dots,c_n][a_0,a_0^{-1}]\cap \mathbb C[a_0, \dots,a_n].
\end{equation}

A general algorithm for these types of problems has been in fact developed by \name{Bigatti-Robbiano} in a recent preprint \cite{BiRo2018Saturations-of-Sub}. It gives by a computer program the $U$-invariants as explicit polynomials up to degree 6.
The complexity of the algorithm, which is general, is much higher than that given by the symbolic method in the special case of $U$-invariants of binary forms with which  those invariants were classically computed.

Roughly speaking the algorithm consists in finding polynomials in the $c_i$ which are divisible by higher and higher powers of $a_0$.

For $n\leq 4$  the algorithm is quite simple and quickly gives:
\subsubsection{The case \texorpdfstring{$n=2$}{n=2}}
$S_2= \CC[  a_0,c_2]$.
\subsubsection{The case \texorpdfstring{$n=3$}{n=3}}
$$ 
8 c_2^3+9 c_3^2 = a_0^2(9 a_0^2 a_3^2-18 a_0 a_1 a_2 a_3+8 a_0 a_2^3+6 a_1^3 a_3-3 a_1^2 a_2^2)= 
a_0^2 D,
$$  
with $D$ of degree $4$ and weight $6$, thus an $\SL(2,\CC)$-invariant (Proposition~\ref{cova}), the {\it discriminant}.   The algorithm stops after this point  and   $S_{3}$ is generated by the elements $a_0,c_2,c_3,D$ modulo the relation $a_0^2D-8c_2^3 - 9c_3^2$:
$$
S_{3}=\CC[a_{0},c_{2},c_{3},D],\quad a_0^2D-8c_2^3 - 9c_3^2=0.
$$
\subsubsection{The case \texorpdfstring{$n=4$}{n=4}}
 $$
2 c_4+ c_2^2= a_0^2( 2 a_0  a_4-2 a_1 a_3+  a_2^2):=a_0^2 B,
$$
with $B$ of degree $2$ and weight $4$, hence an $\SL(2,\CC)$-invariant.
$$
6 c_2 B -  D = -a_0 C \text{ \ with \ }
C := 2   a_2^3 -  6   a_1 a_2 a_3 + 9 a_0  a_3^2 + 6  a_1^2 a_4 - 12 a_0  a_2 a_4,
$$
where $C$ has degree $3$ and weight $6$, hence is an $\SL(2,\CC)$-invariant.   Again, the algorithm stops here, the algebra $S_4$ is generated by the $U$-invariants $a_0,c_2,c_3,B,C$ modulo the relation  
$6 a_0^{2}  c_2 B +  a_0^3 C  - 8 c_2^3 - 9 c_3^2$, and the subalgebra of $\SL(2,\CC)$-invariants is generated by $B$ and $C$.
$$
S_{4} =   \CC[a_{0},c_{2},c_{3},B,C], \quad 6 a_0^{2}  c_2 B + a_0^3 C  - 8 c_2^3 - 9 c_3^2=0.
$$
\begin{remark}\label{noperp}
The computations above show that the indecomposable $U$-invariant $D \in S_{3}$ becomes decomposable in $S_{4}$.
\end{remark}

A modern approach to computations of invariants and covariants for binary forms can be found in the thesis of \name{Mihaela Popoviciu Draisma} \cite{Po2013Invariants-of-bina}.  There one can find also references to classical computations.
\bigskip

\renewcommand{\MR}[1]{}
\bibliography{KP-Bib}

\providecommand{\bysame}{\leavevmode\hbox to3em{\hrulefill}\thinspace}
\providecommand{\MR}{\relax\ifhmode\unskip\space\fi MR }
\providecommand{\MRhref}[2]{%
  \href{http://www.ams.org/mathscinet-getitem?mr=#1}{#2}
}
\providecommand{\href}[2]{#2}
\begin{thebibliography}{Mac85b}

\bibitem[BR19]{BiRo2018Saturations-of-Sub}
Anna~Maria Bigatti and Lorenzo Robbiano, \emph{Saturations of {S}ubalgebras,
  {SAGBI} {B}ases, and {U}-{I}nvariants}, arXiv:1909.10901v2, 2019.

\bibitem[FdB76]{Fa1876Theorie-des-Formes}
F.~de Fa{\`a}~de Bruno, \emph{{T}h\'eorie des {F}ormes {B}inaires}, Librairie
  Brero, Turin, 1876.

\bibitem[Gil27]{Gi1927The-Minimum-Weight}
C.~W. Gilham, \emph{The {M}inimum {W}eight of {I}rreducible q-{A}ry
  {P}erpetuant {T}ypes}, J. London Math. Soc. \textbf{2} (1927), no.~4,
  209--210. \MR{1574818}

\bibitem[Gor68]{Go1868Beweis-dass-jede-C}
Paul Gordan, \emph{Beweis, dass jede {C}ovariante und {I}nvariante einer
  bin{\"a}ren {F}orm eine ganze {F}unction mit numerischen {C}oefficienten
  einer endlichen {A}nzahl solcher {F}ormen ist}, J. Reine Angew. Math.
  \textbf{69} (1868), 323--354. \MR{1579424}

\bibitem[Gor87]{Go1987Vorlesungen-uber-I}
\bysame, \emph{Vorlesungen {\"u}ber {I}nvariantentheorie}, second ed., Chelsea
  Publishing Co., New York, 1987, Erster Band: Determinanten. [Vol. I:
  Determinants], Zweiter Band: Bin{\"a}re Formen. [Vol. II: Binary forms],
  Edited by Georg Kerschensteiner. \MR{917266}

\bibitem[Gra03a]{Gr1903On-Perpetuants}
J.~H. Grace, \emph{On {P}erpetuants}, Proc. Lond. Math. Soc. \textbf{35}
  (1903), 319--331. \MR{1577002}

\bibitem[Gra03b]{Gr1903Types-of-Perpetuan}
\bysame, \emph{Types of {P}erpetuants}, Proc. Lond. Math. Soc. \textbf{35}
  (1903), 107--111. \MR{1576986}

\bibitem[Hil93]{Hi1993Theory-of-algebrai}
David Hilbert, \emph{Theory of algebraic invariants}, Cambridge University
  Press, Cambridge, 1993, Translated from the German and with a preface by
  Reinhard C. Laubenbacher, Edited and with an introduction by Bernd Sturmfels.
  \MR{1266168}

\bibitem[KR84]{KuRo1984The-invariant-theo}
Joseph P.~S. Kung and Gian-Carlo Rota, \emph{The invariant theory of binary
  forms}, Bull. Amer. Math. Soc. (N.S.) \textbf{10} (1984), no.~1, 27--85.
  \MR{722856}

\bibitem[Mac84]{Ma1884On-Perpetuants}
P.~A. MacMahon, \emph{On {P}erpetuants}, Amer. J. Math. \textbf{7} (1884),
  no.~1, 26--46. \MR{1505369}

\bibitem[Mac85a]{Ma1885Memoir-on-Seminvar}
\bysame, \emph{Memoir on {S}eminvariants}, Amer. J. Math. \textbf{8} (1885),
  no.~1, 1--18. \MR{1505405}

\bibitem[Mac85b]{Ma1885A-Second-Paper-on-}
\bysame, \emph{A {S}econd {P}aper on {P}erpetuants}, Amer. J. Math. \textbf{7}
  (1885), no.~3, 259--263. \MR{1505386}

\bibitem[Mac95]{Ma1894The-Perpetuant-Inv}
\bysame, \emph{The {P}erpetuant {I}nvariants of {B}inary {Q}uantics}, Proc.
  Lond. Math. Soc. \textbf{26} (1894/95), 262--284. \MR{1575893}

\bibitem[Nag59]{Na1959On-the-14-th-probl}
Masayoshi Nagata, \emph{On the {$14$}-th problem of {H}ilbert}, Amer. J. Math.
  \textbf{81} (1959), 766--772.

\bibitem[Nag65]{Na1965Lectures-on-the-fo}
\bysame, \emph{Lectures on the fourteenth problem of {H}ilbert}, Tata Institute
  of Fundamental Research, Bombay, 1965. \MR{0215828 (35 \#6663)}

\bibitem[PD13]{Po2013Invariants-of-bina}
Mihaela Popoviciu~Draisma, \emph{Invariants of binary forms}, Ph.D. thesis,
  University of Basel, https://edoc.unibas.ch/33424/, 2013.

\bibitem[Pro07]{Pr2007Lie-Groups---An-Ap}
Claudio Procesi, \emph{Lie {G}roups---{A}n {A}pproach through {I}nvariants and
  {R}epresentations}, Universitext, Springer, New York, 2007. \MR{MR2265844
  (2007j:22016)}

\bibitem[Rot99]{Ro1999Two-turning-points}
Gian-Carlo Rota, \emph{Two turning points in invariant theory}, Math.
  Intelligencer \textbf{21} (1999), no.~1, 20--27. \MR{1665154}

\bibitem[RT94]{RoTa1994The-Classical-Umbr}
Gian-Carlo Rota and B.~D. Taylor, \emph{The {C}lassical {U}mbral {C}alculus},
  {SIAM} J. Math. Anal. \textbf{25} (1994), no.~2, 694--711.

\bibitem[Str90]{St1890Ueber-die-symbolis}
Emil Stroh, \emph{{\"U}ber die symbolische {D}arstellung der {G}rundsyzyganten
  einer bin{\"a}ren {F}orm sechster {O}rdnung und eine {E}rweiterung der
  {S}ymbolik von {C}lebsch}, Math. Ann. \textbf{36} (1890), no.~2, 262--303.
  \MR{1510624}

\bibitem[Syl82]{Sy1882On-Subvariants-i.e}
J.~J. Sylvester, \emph{On {S}ubvariants, i.e. {S}emi-{I}nvariants to {B}inary
  {Q}uantics of an {U}nlimited {O}rder}, Amer. J. Math. \textbf{5} (1882),
  no.~1--4, 79--136. \MR{1505319}

\bibitem[Woo04]{Wo1904On-the-Irreducibil}
P.~W. Wood, \emph{On the {I}rreducibility of {P}erpetuant {T}ypes}, Proc.
  London Math. Soc. (2) \textbf{1} (1904), 480--484. \MR{1576797}

\bibitem[Woo05]{Wo1905Alternative-Expres}
\bysame, \emph{Alternative {E}xpressions for {P}erpetuant {T}ype {F}orms},
  Proc. London Math. Soc. (2) \textbf{3} (1905), 334--344. \MR{1575936}

\bibitem[Woo07]{Wo1907On-the-Reducibilit}
\bysame, \emph{On the {R}educibility of {C}ovariants of {B}inary {Q}uantics of
  {I}nfinite {O}rder}, Proc. London Math. Soc. (2) \textbf{5} (1907), 177--196.
  \MR{1577332}

\bibitem[You24]{Yo1924Ternary-Perpetuant}
Alfred Young, \emph{Ternary {P}erpetuants}, Proc. London Math. Soc. (2)
  \textbf{22} (1924), 171--200. \MR{1575702}

\bibitem[YW05]{YoWo1905Perpetuant-Syzygie}
A.~Young and P.~W. Wood, \emph{Perpetuant {S}yzygies}, Proc. London Math. Soc.
  (2) \textbf{2} (1905), 221--265. \MR{1577271}

\end{thebibliography}
 \bibliographystyle{amsalpha}

\end{document}